\documentclass[12pt,a4paper,reqno]{amsart}
\usepackage[utf8]{inputenc}
\usepackage[T1]{fontenc}
\usepackage{indentfirst}
\usepackage[english]{babel}
\usepackage{enumitem}
\usepackage{csquotes}
\usepackage[style=alphabetic,sorting=nyt,maxcitenames=4,maxbibnames=5]{biblatex}
\usepackage{amsmath,amsfonts,amssymb,amsopn,amscd,amsthm}
\usepackage{mathrsfs,shuffle,bbm}
\usepackage{stmaryrd}
\usepackage[vcentermath]{youngtab}
\usepackage{graphicx}
\usepackage[dvipsnames]{xcolor}
\usepackage[colorlinks=true,citecolor=DarkOrchid,linkcolor=NavyBlue]{hyperref}
\usepackage{pgfplots}
\usepackage{tikz}
\usetikzlibrary{patterns,arrows.meta}
\pgfplotsset{compat=1.17}
\usepackage[a4paper,twoside,margin=0.8in]{geometry}
    \newtheorem{main}{Theorem}
    
    \newtheorem{definition}{Definition}
    \newtheorem{lemma}[definition]{Lemma}
    
    \newtheorem{theorem}[definition]{Theorem}
    \newtheorem{proposition}[definition]{Proposition}
    \newtheorem{corollary}[definition]{Corollary}
    
    \theoremstyle{remark}
    \newtheorem{example}[definition]{Example}
    \newtheorem{remark}[definition]{Remark}

    \makeatletter
    \def\thm@space@setup{\thm@preskip=0.5cm   \thm@postskip=0.5cm}
    \makeatother

\newcommand{\N}{\mathbb{N}}
\newcommand{\Z}{\mathbb{Z}}

\newcommand{\R}{\mathbb{R}}
\newcommand{\C}{\mathbb{C}}

\newcommand{\E}{\mathrm{e}}
\newcommand{\I}{\mathrm{i}}
\newcommand{\esper}{\mathbb{E}}
\newcommand{\var}{\mathrm{var}}

\newcommand{\proba}{\mathbb{P}}

\newcommand{\eps}{\varepsilon}

\newcommand{\lle}{\left[\!\left[} 
\newcommand{\rre}{\right]\!\right]}
\newcommand{\sym}{\mathfrak{S}}
\newcommand{\ym}{\mathfrak{Y}}
\newcommand{\DD}[1]{\,d\hspace{-0.3mm}{#1}}
\newcommand{\scal}[2]{\left\langle #1\vphantom{#2}\,\right |\left.#2 \vphantom{#1}\right\rangle}
\newcommand{\comment}[1]{}
\renewcommand{\Re}{\mathrm{Re}}

\newcommand{\dkol}{d_{\mathrm{Kol}}}
\newcommand{\ST}{\mathrm{ST}}
\newcommand{\SST}{\mathrm{SST}}

\newcommand{\ch}{\mathrm{ch}}
\newcommand{\obs}{\mathscr{O}}
\newcommand{\expn}{^{(n)}}
\newcommand{\majT}{\mathrm{maj}(T)}
\newcommand{\majTn}{\mathrm{maj}(T\expn)}
\newcommand{\Sym}{\mathrm{Sym}}
\newcommand{\tildeX}{\widetilde{X}}

\setlist[enumerate]{itemsep=10pt,topsep=10pt}
\setlist[itemize]{itemsep=5pt,topsep=5pt}

\title[Asymptotics of the major index of a random standard tableau]{Asymptotics of the major index\\ of a random standard tableau}
\author{Pierre-Loïc Méliot and Ashkan Nikeghbali}
\date{\today}

\begin{document}

\begin{abstract} In this article, we establish the mod-$\phi$ convergence of the major index of a uniform random standard tableau whose shape converges in the Thoma simplex. This implies various probabilistic estimates, in particular speed of convergence estimates of Berry--Esseen type, and strong large deviation principles.
\end{abstract}

\maketitle
\bigskip

\tableofcontents

\clearpage

\section{Major index of a random standard tableau}
\subsection{Integer partitions and standard tableaux} If $\lambda = (\lambda_1 \geq \lambda_2 \geq \cdots \geq \lambda_{\ell(\lambda)})$ is a non-increasing sequence of positive integers, we recall that it is called \emph{integer partition} of the integer $n = |\lambda| = \sum_{i=1}^{\ell(\lambda)} \lambda_i$, and that it is represented by a \emph{Young diagram}, which is the array of $n$ boxes with $\lambda_1$ boxes on the first row, $\lambda_2$ boxes on the second row, and so on. For instance, the integer partition $\lambda=(4,2,2,1)$ has size $|\lambda|=9$, and it is represented by the Young diagram:
$$\yng(1,2,2,4)\vspace{2mm}$$
We denote $\ym(n)$ the set of integer partitions with size $n$, and $\ym=\bigsqcup_{n \in \N} \ym(n)$ the set of all integer partitions. Given $\lambda \in \ym(n)$, a \emph{standard tableau} with shape $\lambda$ is a numbering of the cells of the Young diagram of $\lambda$ by the integers in $\lle 1,n\rre$ which is bijective and strictly increasing along the rows and columns. For instance,
$$\young(8,47,35,1269)\vspace{2mm}$$
is a standard tableau with shape $(4,2,2,1)$. The set of standard tableaux with shape $\lambda$ will be denoted $\ST(\lambda)$, and the cardinality of this set is given by the Frame--Robinson--Thrall \emph{hook length formula} \cite{FRT54}. Given a cell $\oblong$ of a Young diagram $\lambda$, its hook length $h(\oblong)$
 is the number of cells of the hook which is included in the Young diagram $\lambda$ and which contains $\oblong$, the cells above $\oblong$ and the cells on the right of $\oblong$. For instance, the integer partition $\lambda=(4,2,2,1)$ yields the following hook lengths:
$$\young(1,31,42,7521)\vspace{2mm}$$
Then, for any integer partition $\lambda \in \ym(n),$
$$|\ST(\lambda)| = \frac{n!}{\prod_{\oblong \in \lambda} h(\oblong)};$$
see for instance \cite[Section I.4, Example 3]{Mac95}. In this article, we shall be interested in the distribution of a certain statistics $X : \ST(\lambda) \to \N$, the set $\ST(\lambda)$ being endowed with the uniform law. Before presenting this statistics, let us associate to the cells of a Young diagram $\lambda$ another integer: if the cell $\oblong$ is in $i$-th row and the $j$-th column of the Young diagram $\lambda$, its \emph{content} is $c(\oblong) = j-i$. For instance, $\lambda=(4,2,2,1)$ admits the following contents:
\newcommand{\mone}{\begin{tikzpicture}[scale=1,baseline=-1.3mm]
\draw [thick] (0,0) -- (0.15,0);
\end{tikzpicture} 1}
\newcommand{\mtwo}{\begin{tikzpicture}[scale=1,baseline=-1.3mm]
\draw [thick] (0,0) -- (0.15,0);
\end{tikzpicture} 2}
\newcommand{\mthree}{\begin{tikzpicture}[scale=1,baseline=-1.3mm]
\draw [thick] (0,0) -- (0.15,0);
\end{tikzpicture} 3}
$$\young({\mthree},{\mtwo}{\mone},{\mone}0,0123)$$
The contents are involved in numerous combinatorial formul{\ae} in the theory of integer partitions. For instance, the cardinality of the set $\SST(\lambda,m)$ of semistandard tableaux with shape $\lambda$ and entries in $\lle 1,m\rre$ is given by the product $\prod_{\oblong \in \lambda} \frac{m+c(\oblong)}{h(\oblong)}$; see \cite[Corollary 7.21.4]{Stan99}. The contents of the cells of the Young diagrams shall also play an important role in our work.
\medskip

\subsection{Major index and Schur functions}\label{sub:schur_functions}
A \emph{descent} of a standard tableau of size $n$ is an integer $i \in \lle 1,n-1\rre$ such that $i+1$ appears in a row strictly above the row containing $i$. The \emph{major index} of a standard tableau $T$ is the sum of its descents:
$$\majT = \sum_{i \in \mathrm{Desc}(T)} i.$$
For instance, if $\lambda=(4,2,2,1)$ and $T$ is the standard tableau given as an example in the previous paragraph,
then the set of descents of $T$ is $\{2,3,6,7\}$, and the major index is $2+3+6+7 = 18$. The definitions of descent and major index for a standard tableau are closely related to the analogue definitions for a permutation. Given $\sigma \in \sym(n)$, a descent of the permutation $\sigma$ is an index $i \in \lle 1,n-1\rre$ such that $\sigma(i)>\sigma(i+1)$. For instance, the descent set of the permutation $\sigma=592138647$ is $\{2,3,6,7\}$. The \emph{Robinson--Schensted--Knuth algorithm} yields a bijection between $\sym(n)$ and the set $\bigsqcup_{\lambda \in \ym(n)} \ST(\lambda)\times \ST(\lambda)$ of pairs of standard tableaux with the same shape. For instance, $592138647 \in \sym(9)$ corresponds to the pair of standard tableaux
$$P = \young(9,58,26,1347)\quad;\quad Q = \young(8,47,35,1269).$$
This bijection preserves the set of descents: if $\mathrm{RSK}(\sigma)=(P,Q)$, then $\mathrm{Desc}(\sigma) = \mathrm{Desc}(Q)$; see for instance \cite[Theorem 10.117]{Loehr11}.\medskip

The purpose of this article is to study the distribution of $\majT$ when $T$ is taken uniformly at random in the set $\ST(\lambda)$ of standard tableaux with shape $\lambda$, and $\lambda=\lambda\expn$ is an integer partition with size $n$, $n$ going to infinity. Equivalently, this amounts to consider the distribution of $\mathrm{maj}(\sigma)$ when $\sigma$ is a permutation taken uniformly at random in a RSK class of growing size. The case of a uniform random permutation in $\sym(n)$ can essentially be recovered as a particular case of our results; see the discussion after the statement of Theorems \ref{main:log_laplace} and \ref{main:large_deviations}. The following result due to Billey--Konvalinka--Swanson \cite{BKS20} ensures that under mild hypotheses and after appropriate scaling, this distribution of $\majT$ with $T \sim \mathcal{U}(\ST(\lambda))$ is asymptotically normal:

\begin{theorem}[Billey--Konvalinka--Swanson]\label{thm:BKS}
We denote $\lambda'$ the conjugate of an integer partition $\lambda$, which is obtained by symmetrizing the Young diagram with respect to the diagonal. Consider a sequence of integer partitions $(\lambda\expn)_{n \geq 1}$ such that $$|\lambda\expn|=n\quad;\quad n-\lambda\expn_1 \to +\infty \quad;\quad n-\lambda^{(n)'}_1 \to +\infty.$$
 Then, with $T\expn$ taken uniformly at random in $\ST(\lambda\expn)$,
$$\frac{\majTn - \esper[\majTn]}{\sqrt{\var(\majTn)}} \rightharpoonup_{n \to +\infty} \mathcal{N}(0,1).$$
\end{theorem}

\noindent The conditions $ n-\lambda\expn_1 \to +\infty $ and $n-\lambda^{(n)'}_1 \to +\infty$ are actually necessary: otherwise, the limiting distribution exists but is not Gaussian, see \cite[Theorem 1.3]{BKS20}. We shall see in the sequel that if $|\lambda|=n$ and $T \sim \mathcal{U}(\ST(\lambda))$, then 
$$ \sum_{i=1}^{\ell(\lambda)} (i-1)\lambda_i \leq \majT \leq \binom{n}{2} - \sum_{i=1}^{\ell(\lambda)} \binom{\lambda_i}{2}$$
and the variance of $\majT$ is of order $O(n^3)$. Thus, in most cases, the values of $\majT$ are of order $O(n^2)$, and their fluctuations around the mean are of order $O(n^{\frac{3}{2}})$ and asymptotically normal. This raises the question of the large deviations of $\majT$ of order $O(n^2)$: given $y > 0$, what are the asymptotics of
$$\proba[\majTn - \esper[\majTn] \geq yn^2]\,?$$
We shall explain in this article how to estimate these probabilities if one knows the limit shape of the partitions $\lambda\expn$. On the other hand, Theorem \ref{thm:BKS} ensures that
\begin{align*}
&\dkol\!\left(\frac{\majTn - \esper[\majTn]}{\sqrt{\var(\majTn)}},\,\mathcal{N}(0,1)\right) \\
&= \sup_{s \in \R} \left|\proba\!\left[\majTn - \esper[\majTn] \leq s\sqrt{\var(\majTn)} \right] - \int_{-\infty}^s \E^{-\frac{x^2}{2}}\,\frac{\!\DD{x}}{\sqrt{2\pi}}\right|\end{align*}
goes to $0$ when $n$ goes to infinity, but it does not give the speed of convergence, related to the quality of the Gaussian approximation. One of our main results is a uniform upper bound for this Kolmogorov distance, which is of order $O(n^{-\frac{1}{2}})$.\medskip

The major index of standard tableaux is related to the Schur functions and to the so-called \emph{$q$-hook length formula}. For any family of variables $(x_1,\ldots,x_N)$ with $N \geq \ell(\lambda)$, set 
$$s_\lambda(x_1,\ldots,x_N) = \frac{\det((x_i)^{\lambda_j+N-j})_{1\leq i,j\leq N}}{\det((x_i)^{N-j})_{1\leq i,j\leq N}}$$
with by convention $\lambda_j=0$ if $j>\ell(\lambda)$. 
The specialisation $x_{N+1}=0$ from $\R[x_1,\ldots,x_{N+1}]$ to $\R[x_1,\ldots,x_N]$ stabilises these \emph{Schur polynomials}:
$$\forall N \geq \ell(\lambda),\,\,\,s_\lambda(x_1,\ldots,x_N,0) = s_\lambda(x_1,\ldots,x_N).$$
In the projective limit in the category of graded algebras $\R[X] = \varprojlim_{N \to \infty}\R[x_1,\ldots,x_N]$, there exists a unique element $s_\lambda(X)$ whose projections are the symmetric polynomials defined above. This is the \emph{Schur function} $s_\lambda$, see for instance \cite[Section I.3]{Mac95}. An important result in the theory of these symmetric functions is the Stanley $q$-hook length formula \cite[Corollary 7.21.3]{Stan99}: if $|q|<1$, then
$$s_\lambda(1,q,q^2,\ldots,q^n,\ldots) = q^{b(\lambda)}\,\prod_{\oblong \in \lambda} \left(\frac{1}{1-q^{h(\oblong)}}\right),$$
with $b(\lambda)=\sum_{i=1}^{\ell(\lambda)}(i-1)\lambda_i$. Moreover, up  to a combinatorial factor, the principal specialisation $s_\lambda(1,q,q^2,\ldots)$ is the generating series of the major indices of the standard tableaux with shape $\lambda$:
$$s_\lambda(1,q,q^2,\ldots) = \left(\prod_{i=1}^{|\lambda|}\frac{1}{1-q^i}\right)\,\sum_{T \in \ST(\lambda)}q^{\majT},$$
see \cite[Proposition 7.19.11]{Stan99}.
\medskip

\subsection{Cumulants of the major index}
The starting point of the argument of \cite{BKS20} is the following important remark, which is based on earlier
works by Chen--Wang--Wang \cite{CWW08} and by Hwang--Zachavoras \cite{HZ15}. Given $T \in \ST(\lambda)$, set $X(T) = \majT - b(\lambda)$, and consider the generating series of this statistics:
$$\sum_{T \in \ST(\lambda)} q^{X(T)} = \frac{\prod_{i=1}^{|\lambda|}(1-q^i)}{\prod_{\oblong \in \lambda}(1-q^{h(\oblong)})}.$$
We obtain a product of ratios of $q$-integers $[a]_q = \frac{1-q^a}{1-q}$. This leads to simple formulas for the \emph{cumulants} of the random variable $X(T)$ with $T\sim \mathcal{U}(\ST(\lambda))$. Given a random variable whose Laplace transform $\esper[\E^{zX}]$ is convergent on a disk $D_{(0,R)} = \{z \in \C\,\,|\,\,|z|<R\}$ with $R>0$, we recall that the cumulants of $X$ are the coefficients $\kappa^{(r)}(X)$ of the log-Laplace transform:
$$\log \esper[\E^{zX}] = \sum_{r=1}^\infty \frac{\kappa^{(r)}(X)}{r!}\,z^r\quad\text{in a neighborhood of }0.$$
The cumulants of $X$ are related to its moments by a Möbius inversion formula with respect to the lattice of set partitions:
$$\kappa^{(r)}(X) = \sum_{\pi \in \mathfrak{P}(r)} (-1)^{\ell(\pi)-1}\,(\ell(\pi)-1)!\, \left(\prod_{j=1}^{\ell(\pi)} \esper[X^{|\pi_j|}]\right),$$
the sum running over the set $\mathfrak{P}(r)$ of set partitions $\pi_1 \sqcup \pi_2 \sqcup \cdots \sqcup \pi_{\ell(\pi)}$ of the integer interval $\lle 1,r\rre$. This combinatorial formula enables one to define the cumulants of $X$ assuming only that $X$ has moments of all order. Now, consider a finite set $\mathfrak{T}$ endowed with the uniform distribution, and a statistics $X : \mathfrak{T} \to \N$. The generating function $\esper[q^X]$ is given by the ratio $\frac{P(q)}{P(1)}$, where 
$$P(q) = \sum_{m=0}^\infty \big|\{T \in \mathfrak{T}\,\,|\,\,X(T)=m\}\big|\,q^m \qquad;\qquad P(1) = \sum_{m=0}^\infty \big|\{T \in \mathfrak{T}\,\,|\,\,
X(T)=m\}\big| = |\mathfrak{T}|.$$
\begin{lemma}\label{lem:fundamental}
Suppose that the polynomial $P \in \Z[q]$ is given by a ratio of $q$-integers:
$$P(q) = \frac{\prod_{k=1}^l[b_k]_{q}}{\prod_{k=1}^l[a_k]_{q}},$$
where $\{a_1,\ldots,a_l\}$ and $\{b_1,\ldots,b_l\}$ are multisets of non-negative integers. Then, $X$ takes its values in $\lle 0,n\rre$ with $n=\sum_{k=1}^l (b_k-a_k)$, and for any $r \geq 1$,
$$\kappa^{(r)}(X) = \frac{B_r}{r}\,\sum_{k=1}^l ((b_k)^r-(a_k)^r),$$
where $B_r$ is the $r$-th Bernoulli number. Consequently, the odd cumulants of $X$ of order larger than $3$ vanish.
\end{lemma}

\begin{proof}
Denote $c_m = \big|\{T \in \mathfrak{T}\,\,|\,\,X(T)=m\}\big|$, so that $P(q) = \sum_{m=0}^n c_m\,q^m$. By hypothesis, $P$ is a polynomial of degree $n$ and is the ratio of two polynomials $\prod_{k=1}^l (1-q^{b_k})$ and $\prod_{k=1}^l (1-q^{a_k})$ with degrees $\sum_{k=1}^l b_k$ and $\sum_{k=1}^l a_k$, so $n$ is equal to the difference of the two sums. Now, the Bernoulli numbers $B_r$ are defined by their generating series
$$\frac{t}{1-\E^{-t}}=\sum_{r=0}^\infty B_r\,\frac{t^r}{r!} = 1 + t\,\frac{\partial}{\partial t} \left(\sum_{r=1}^\infty\frac{B_r}{r}\, \frac{t^r}{r!}\right),$$
so the \emph{divided Bernoulli numbers} $\frac{B_r}{r}$ have for exponential generating series:
$$\sum_{r=1}^\infty\frac{B_r}{r}\, \frac{t^r}{r!} = \log\left(\frac{\E^t-1}{t}\right). $$
It suffices now to write, for $q=\E^z$ with $z$ close to $0$:
\begin{align*}
\log \esper[\E^{zX}] &= \sum_{k=1}^l \log\left(\frac{\E^{zb_k}-1}{\E^{za_k}-1}\right) = \sum_{k=1}^l \log\left(\frac{\E^{zb_k}-1}{zb_k}\right)-\log\left(\frac{\E^{za_k}-1}{za_k}\right) \\ 
&= \sum_{r=1}^\infty \frac{B_r}{r}\,z^r \left(\sum_{k=1}^l((b_k)^r - (a_k)^r)\right).
\end{align*}
Since $\frac{t}{1-\E^{-t}} - \frac{t}{2} = t\,\coth (\frac{t}{2})$ is an even function, the odd Bernoulli numbers of order $2r+1 \geq 3$ all vanish. Therefore, $\kappa^{(2r+1)}(X)=0$ for any $r \geq 1$.
\end{proof}
\medskip

Let us discuss a bit the properties of the generating series $\varphi(z) = \sum_{r=1}^\infty \frac{B_r}{r}\,\frac{z^r}{r!} = \log(\frac{\E^z-1}{z})$ of the divided Bernoulli numbers. We shall also work with 
$$\phi(z)=\varphi(z)-\frac{z}{2}=\log\left(\frac{\sinh \frac{z}{2}}{\frac{z}{2}}\right) = \sum_{r=2}^\infty \frac{B_r}{r}\,\frac{z^r}{r!},$$
which is an even function. The Bernoulli numbers are related to the Riemann $\zeta$ function by the equation:
$$B_{r} = 2\,\frac{(-1)^{\frac{r}{2}+1}\,r!}{(2\pi)^r}\,\zeta(r)$$
for $r$ even. Therefore, $|B_r| \simeq 4\sqrt{\pi r} \left(\frac{r}{2\pi \E}\right)^r$ for $r$ even, and the radius of convergence of the series $\varphi$  and $\phi$ is equal to $2\pi$. Indeed, these series exhibit two singularities at $t=\pm 2\I \pi$. However, the open disc $D_{(0,2\pi)}$ is not the whole domain of definition and analyticity of these functions. Indeed, $\phi'(z) =  \frac{1}{2} \coth \frac{z}{2}  - \frac{1}{z}$ makes sense for any $z \notin 2\I \pi \Z$, so if $z \notin \I\R$, then we can set 
$$\phi(z) = \int_{0}^1  \left(\frac{z}{2}\coth \frac{tz}{2} -\frac{1}{t}\right)\DD{t}.$$
Thus, the functions $\varphi(z)$ and $\phi(z)$ admit analytic extensions to the domain 
$$\mathscr{D}_0 = \C \setminus (\I[2\pi,+\infty) \sqcup \I(-\infty,-2\pi]).$$
In several proofs hereafter, an important property of the domain $\mathscr{D}_0$ will be that it is stable by the operation $z \mapsto xz$ for any $x \in [-1,1]$.
\begin{figure}[ht]
\begin{center}        
\begin{tikzpicture}[scale=0.75]
\draw [->] (-4,0) -- (4,0);
\draw [->] (0,-3) -- (0,3);
\draw [red,very thick] (0,1.5) -- (0,2.9);
\draw [red,very thick] (0,-1.5) -- (0,-3);
\draw [red,very thick] (-0.1,1.5) -- (0.1,1.5);
\draw [red,very thick] (-0.1,-1.5) -- (0.1,-1.5);
\draw [red] (-0.5,1.5) node {$2\I\pi$};
\draw [red] (-0.7,-1.5) node {$-2\I\pi$};
\end{tikzpicture}
\caption{The domain of definition and analyticity of the series $\phi(z)$ and $\varphi(z)$.}
\end{center}
\end{figure}
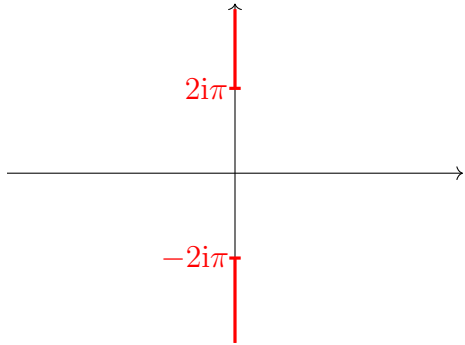

With  $\mathfrak{T} = \ST(\lambda)$ and $X(T)=\majT-b(\lambda)$, the random variable $X$ satisfies the hypotheses of Lemma \ref{lem:fundamental} with:
$$|\ST(\lambda)|\,\esper[q^X] = P(q) = \frac{\prod_{i=1}^{|\lambda|}[i]_q}{\prod_{\oblong \in \lambda}[h(\oblong)]_q}.$$
Therefore, for any $r \geq 1$,
\begin{equation}
    \kappa^{(r)}(\majT - b(\lambda)) = \frac{B_r}{r}\,\left(\sum_{i=1}^{|\lambda|} i^r - \sum_{\oblong \in \lambda} (h(\oblong))^r\right).\label{eq:cumulant_maj_ST}
\end{equation}
A large part of our work will consist in analysing the asymptotics of this combinatorial formula when $\lambda$ is a partition of growing size. If the growth of the sequence $(\lambda\expn)$ is specified a bit more precisely than in Theorem \ref{thm:BKS}, then it is a possible to write an asymptotic expansion of each $r$-th cumulant, which leads to an asymptotic expansion of the scaled log-Laplace transform $$\log \esper[\E^{z\,\frac{\majTn}{n}}];$$
see Theorem \ref{main:log_laplace}. Then, standard arguments of probability theory enable the computation of the large deviation estimates and of the Kolmogorov distance (see Theorems \ref{main:large_deviations} and \ref{main:berry_esseen} hereafter). 
\medskip

\subsection{Growing partitions and the Thoma simplex}\label{sub:thoma_simplex}
We are interested in the asymptotic behavior of the random variable $\majT$ when the underlying Young diagram $\lambda=\lambda\expn$ has size $n$ and grows to infinity while having a certain limit shape. There are several notions of limit shapes for Young diagrams; in our setting, a natural assumption is that the rows $\lambda\expn_1,\lambda\expn_2,\ldots$ and the columns $\lambda^{(n)'}_1,\lambda^{(n)'}_2,\ldots$ grow with known asymptotic frequencies $(\alpha_i)_{i \geq 1}$ and $(\beta_i)_{i \geq 1}$:
$$\alpha_i = \lim_{n \to \infty} \frac{\lambda\expn_i}{n}\qquad;\qquad \beta_i = \lim_{n \to \infty} \frac{\lambda^{(n)'}_i}{n}.$$
It is then convenient to introduce the so-called \emph{Frobenius coordinates} of Young diagrams, and to use them in order to embed $\ym = \bigsqcup_{n \in \N} \ym(n)$ in the \emph{Thoma simplex}. If $\lambda$ is an integer partition with size $n$, its Frobenius coordinates $(a_1,a_2,\ldots,a_d\,|\,b_1,b_2,\ldots,b_d)$ are the two sequences of half-integers which measure the size of the rows and columns of $\lambda$, starting from the diagonal; see Figure \ref{fig:frobenius_coordinates} for an example in size $11$, with $d=2$. 
\begin{figure}[ht]
 \begin{center}        
$$\vspace{2mm}\lambda = \begin{tikzpicture}[scale=1,baseline=0.5cm]
\draw [thick,red,dashed] (0,0) -- (2,2);
\draw (0,0) -- (5,0) -- (5,1) -- (4,1) -- (4,2) -- (2,2) -- (2,3) -- (0,3) -- (0,0);
\draw (4,0) -- (4,1) -- (0,1);
\draw (3,0) -- (3,2);
\draw (2,0) -- (2,2) -- (0,2);
\draw (1,0) -- (1,3);
\draw [thick,<->] (0.6,0.5) -- (4.9,0.5);
\draw [thick,<->] (1.6,1.5) -- (3.9,1.5);
\draw [thick,<->] (0.5,0.6) -- (0.5,2.9);
\draw [thick,<->] (1.5,1.6) -- (1.5,2.9);
\draw (1.7,2.3) node {\tiny $b_2$};
\draw (0.7,2.3) node {\tiny $b_1$};
\draw (3.4,1.7) node {\tiny $a_2$};
\draw (3.4,0.7) node {\tiny $a_1$};
\end{tikzpicture} =\left(\frac{9}{2},\frac{5}{2}\,\right|\left. \frac{5}{2},\frac{3}{2}\right).$$ 
 \caption{Frobenius coordinates of the integer partition $\lambda=(5,4,2)$.\label{fig:frobenius_coordinates}}
 \end{center}
 \end{figure}
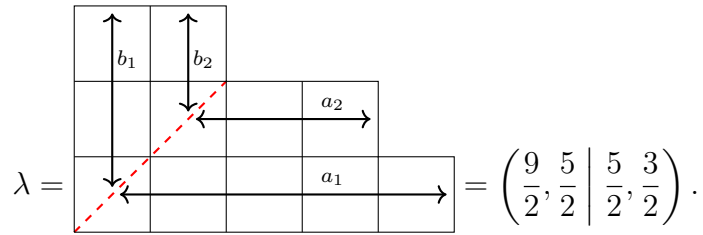

By considering the Frobenius coordinates as areas of regions of the Young diagram, it is easy to see that the sum $\sum_{i=1}^d (a_i+b_i)$ is equal to the size of the partition $\lambda$. The Thoma simplex $\Omega$ is defined as the set of pairs of infinite nonincreasing sequences
$$\Omega = \left\{((\alpha_i)_{i \geq 1},(\beta_i)_{i \geq 1})\,\,\bigg|\,\,\alpha_1\geq \alpha_2 \geq \cdots \geq 0,\,\,\beta_1 \geq \beta_2 \geq \cdots \geq 0,\,\,\sum_{i=1}^\infty (\alpha_i +\beta_i) \leq 1\right\}.$$
This infinite-dimensional simplex plays an important role in the asymptotic representation theory of the symmetric groups, because it parametrises the extremal characters of $\sym(\infty)$; see \cite{Tho64,KV81} and \cite[Theorem 11.31]{Mel17}. Any integer partition $\lambda=(a_1,\ldots,a_d\,|\,b_1,\ldots,b_d)$ with size $n \geq 1$ can be seen as an element of the Thoma simplex, by associating to it the pair $\omega_\lambda=(\alpha,\beta)$ with
$$\alpha = \left(\frac{a_1}{n},\frac{a_2}{n},\ldots,\frac{a_d}{n},0,0,\ldots\right)\quad;\quad \beta = \left(\frac{b_1}{n},\frac{b_2}{n},\ldots,\frac{b_d}{n},0,0,\ldots\right).$$
We call \emph{growing} a sequence of integer partitions $(\lambda\expn)_{n \geq 1}$ with $|\lambda\expn|=n$ for any $n \geq 1$, and we say that it is \emph{convergent} if it is growing and if the Thoma parameters $\omega_{\lambda\expn}$ converge coordinatewise towards a parameter $\omega \in \Omega$. Equivalently, all the rows and columns of the partitions $\lambda\expn$ rescaled by a factor $\frac{1}{n}$ admit limit frequencies.
Notice that the sum of the coordinates of the Thoma parameter $\omega_\lambda$ of an integer partition is always equal to
$$\frac{a_1}{n}+\cdots + \frac{a_d}{n} + \frac{b_1}{n}+\cdots + \frac{b_d}{n} = \frac{n}{n}
 = 1.$$
By the Fatou lemma, the sum of the coordinates of a limit $\omega$ of Thoma parameters of integer partitions is  smaller than $1$, and it can be strictly smaller than $1$ (consider for instance the case where $\lambda\expn$ is a square of size $\lfloor \sqrt{n} \rfloor \times \lfloor \sqrt{n}\rfloor$, plus a $O(\sqrt{n})$ additional boxes on its first row; then, the limit $\omega$ is the pair of null sequences).

\begin{lemma}\label{lem:frobenius_thoma}
Given a sequence $(\omega\expn)_{n \geq 1}=(\alpha\expn,\beta\expn)_{n \geq 1}$ of Thoma parameters, the following are equivalent:
\begin{enumerate}
    \item The parameters $\omega\expn$ converge coordinatewise towards a parameter $\omega$.
    \item For any $k\geq 2$, the $k$-th moments
    $$p_k(\omega\expn) = \sum_{i=1}^\infty (\alpha\expn_i)^k + (-1)^{k-1}\sum_{i=1}^\infty (\beta\expn_i)^k $$
    converge towards $p_k(\omega)$.
\end{enumerate}
\end{lemma}

\begin{proof}
Given a parameter $\omega=(\alpha,\beta) \in \Omega$, we set $\gamma = 1-\sum_{i=1}^\infty(\alpha_i+\beta_i)$. Let $\mathscr{M}^1([-1,1])$ be the set of Borel probability measures on $[-1,1]$, endowed with the topology of weak convergence. By using standard arguments from the theory of weak convergence of probability measures (see for instance \cite{Bill99}), it is not difficult to show that the map
\begin{align*}
\Omega &\to \mathscr{M}^1([-1,1]) \\
(\alpha,\beta) &\mapsto \mu_{(\alpha,\beta)}=\sum_{i=1}^\infty \alpha_i\,\delta_{\alpha_i} + \sum_{i=1}^\infty \beta_i\, \delta_{-\beta_i} + \gamma \,\delta_0
\end{align*}
is a homeomorphism towards a closed subset of $\mathscr{M}^{1}([-1,1])$. Since $[-1,1]$ is a compact interval, the weak convergence of its probability measures is equivalent to the convergence of all the moments. The result follows since
$$\int_{-1}^1 x^k \,\mu_{\omega}(\!\DD{x}) = p_{k+1}(\omega)$$
for any $k \geq 1$ and any $\omega \in \Omega$.
\end{proof}

In the sequel, we denote $p_1(\omega)=1$ for any $\omega \in \Omega$, and if $\lambda=(a_1,\ldots,a_d\,|\,b_1,\ldots,b_d)$ is an integer partition with size $n = |\lambda|\geq 1$, we set
$$p_k^{\mathrm{F}}(\lambda) = n^k\,p_k(\omega_\lambda) = \sum_{i=1}^d (a_i)^k + (-1)^{k-1} \sum_{i=1}^d (b_i)^k.$$
The $p_k^{\mathrm{F}}$ with $k \geq 1$ are the \emph{Frobenius moments}; in particular, $p_1^{\mathrm{F}}(\lambda)=|\lambda|=n$. One of the main tool that we shall use is the following:

\begin{theorem}[Asymptotics of the cumulants of the major index]\label{thm:asymptotics_cumulants}
Let $(\lambda\expn)_{n \geq 1}$ be a growing sequence of integer partitions; we denote $(T\expn)_{n \geq 1}$ the associated sequence of random standard tableaux. We have for any $r \geq 2$:
\begin{equation}
\kappa^{(r)}(\majTn) =  \frac{B_r}{r(r+1)}\,((p_1^{\mathrm{F}})^{r+1}-p_{r+1}^{\mathrm{F}} ) + \frac{B_r}{2r} \left((p_1^{\mathrm{F}})^r  + \sum_{s=1}^{r-1} \binom{r}{s} (-1)^{s}\, p_s^{\mathrm{F}}\,p_{r-s}^{\mathrm{F}}\right)  + O(n^{r-1}),\label{eq:cumulant_maj_ST2}
\end{equation}
where $p_k^{\mathrm{F}} = p_k^{\mathrm{F}}(\lambda\expn) = n^k\,p_k(\omega_{\lambda\expn})$, and where the remainder is uniform with respect to the choice of a growing sequence.
\end{theorem}

At first sight, this might seem easy to prove, because the Frobenius coordinates of $\lambda$ dictate the geometry of its Young diagram, and because Equation \eqref{eq:cumulant_maj_ST} relates the cumulants of $\majT$ to the hook lengths of the cells of the Young diagram of $\lambda$, which also seem to be geometric observables. Unfortunately, the geometric connection between $(a_1,\ldots,a_d\,|\,b_1,\ldots,b_d)$ and $\{h(\oblong)\,\,|\,\,\oblong \in \lambda\}$ is more subtle than what one might think, for the following reason. If the cell $\oblong$ is on the $i$-th row and $j$-th column and if it belongs to the square of size $d \times d$ in the bottom left corner of $\lambda$ ($i \leq d$ and $j \leq d$), then $h(\oblong) = a_i + b_j$. 
\begin{figure}[ht]
\begin{center}

\begin{tikzpicture}[scale=1,baseline=0.5cm]
\fill [white!80!black] (0,0) rectangle (2,2);
\draw [thick] (0,0) rectangle (2,2);
\draw [thick,red,dashed] (0,0) -- (2,2);
\draw (0,0) -- (5,0) -- (5,1) -- (4,1) -- (4,2) -- (2,2) -- (2,3) -- (0,3) -- (0,0);
\draw (4,0) -- (4,1) -- (0,1);
\draw (3,0) -- (3,2);
\draw (2,0) -- (2,2) -- (0,2);
\draw (1,0) -- (1,3);
\draw [thick,->] (0.5,1.5) -- (-1.5,1.5);
\draw (-4.5,1.95) node {the part of $\lambda$ where the connection};
\draw (-4.5,1.5) node {between hook lengths and};
\draw (-4.5,1.1) node {Frobenius coordinates is clear};

\end{tikzpicture}
\caption{The Frobenius coordinates are not directly related to the hook lengths.}
\end{center}
\end{figure}
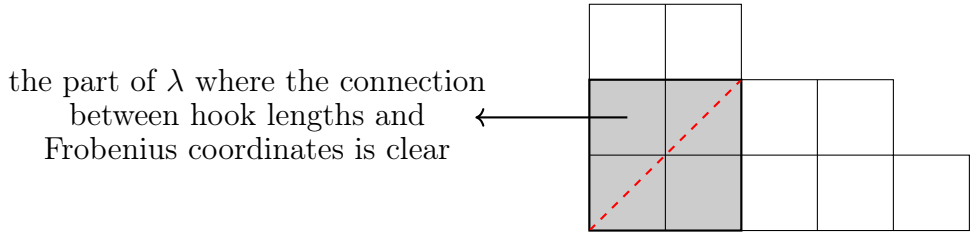
\vspace{2mm}

\noindent However, if $\oblong$ does not belong to the bottom left square (and in general this is the case for a large proportion of boxes), then the relation between $h(\oblong)$ and the Frobenius coordinates is much more complicated. For instance, if $\oblong=(i,j)$ satisfies $i \leq d$ and $j>d$, then $h(\oblong)$ depends on $a_i$ and on which $i'>i$ satisfy $a_{i'}+i'+\frac{1}{2}\geq j$. Thus, the transformation of Equation \eqref{eq:cumulant_maj_ST} into Equation \eqref{eq:cumulant_maj_ST2} (which is suitable to asymptotic analysis) is not immediate. We shall solve this problem by using a combinatorial bijection from the theory of integer partitions (see Lemma \ref{lem:subtle_geometry}), and by making calculations in the so-called Kerov--Olshanski algebra of observables.
\medskip

\subsection{Main results and outline of the paper}
We are now ready to state our main results regarding the asymptotic behavior of $\majTn$ when $T\expn \sim \mathcal{U}(\ST(\lambda\expn))$ and $(\lambda\expn)_{n \geq 1}$ is a growing or convergent sequence of integer partitions. Given a parameter $\omega \in \Omega$, we denote $\mu_\omega$ the corresponding probability measure on $[-1,1]$ (see Lemma \ref{lem:frobenius_thoma}), and  
\begin{align*}
\Lambda_\omega(z) &= \int_{t=0}^1 \int_{x=-1}^1 (\phi(tz) - \phi(txz))\,\mu_\omega(\!\DD{x})\DD{t} ; \\ 
\Psi_\omega(z) &= \frac{1}{2}\int_{x=-1}^1 \int_{y=-1}^1 \left(\phi(z)+\frac{\phi((y-x)z)-\phi(yz)-\phi(-xz)}{xy}\right)\mu_\omega^{\otimes 2}(\!\DD{x}\DD{y}).
\end{align*}

\begin{main}[Asymptotic expansion of the log-Laplace transform]\label{main:log_laplace}
Let $(\lambda\expn)_{n \geq 1}$ be a growing sequence of integer partitions, $(T\expn)_{n \geq 1}$ the associated sequence of random standard tableaux, and $(\mu\expn)_{n \geq 1}$ the associated sequence of probability measures on $[-1,1] $. If $X\expn = \frac{\majTn}{n}$, then
$$
\log \esper\!\left[\E^{z\,(X\expn-\esper[X\expn])}\right] = n\,\Lambda_{\omega\expn}(z) + \Psi_{\omega\expn}(z) + o(1).
$$
The remainder $o(1)$ goes to $0$ uniformly with respect to the growing sequence $(\lambda\expn)_{n \geq 1}$, and uniformly on any compact subset of the domain $\frac{1}{2}\mathscr{D}_0$.
\end{main}

\begin{main}[Strong large deviations of the major index]\label{main:large_deviations}
Let $(\lambda\expn)_{n \geq 1}$ be a convergent sequence of integer partitions with associated Thoma parameters $(\omega\expn)_{n \geq 1}$ and limit parameter $\omega \in \Omega$. We suppose that $\omega$ is not one of the two pairs
$$\omega_1 = ((1,0,\ldots),(0,\ldots))\quad;\quad \omega_{-1} = ((0,\ldots),(1,0,\ldots)).$$
\begin{enumerate}
\item The function $\Lambda_\omega$ is even and strictly convex on $\R$.

\item Fix a real number 
$$0<y<\frac{1}{4}\left(1-\int_{x=-1}^1 x\,\mu_\omega(\!\DD{x})\right).$$
We denote $\Lambda_{\omega\expn}^{*}(y) = \sup_{h \in \R} (hy-\Lambda_{\omega\expn}(h))$ the Legendre--Fenchel conjugate of $\Lambda_{\omega\expn}$. There exists a unique parameter $h \in (0,+\infty)$ such that $h = \Lambda_{\omega}'(y)$, and 
$$\proba[\majTn - \esper[\majTn] \geq  yn^2] = \frac{\E^{-n\,(\Lambda_{\omega\expn})^{*}(y)}}{h \sqrt{2\pi n \Lambda_\omega''(h)}}\,\E^{\Psi_\omega(h)}\,(1+o(1)).$$

\item In particular, with the same assumptions on $y$ and $(\lambda\expn)_{n \geq 1}$,
$$\lim_{n \to \infty} \frac{\log \proba[\majTn - \esper[\majTn] \geq yn^2]}{n} = -\Lambda_\omega^*(y).$$    
\end{enumerate}
\end{main}

It will be clear from the proof of Theorem \ref{main:large_deviations} that similar estimates hold for the probabilities $\proba[\majTn - \esper[\majTn] \leq -yn^2]$: with the same assumptions on $y$, there exists a unique parameter $h_- \in (-\infty,0)$ such that $h_- = \Lambda_\omega'(-y)$, and we have 
$$\proba[\majTn - \esper[\majTn] \leq  -yn^2] = \frac{\E^{-n\,(\Lambda_{\omega\expn})^{*}(-y)}}{|h_-| \sqrt{2\pi n \Lambda_\omega''(h_-)}}\,\E^{\Psi_\omega(h_-)}\,(1+o(1)).$$
The third part of Theorem \ref{main:large_deviations} can also be restated as follows: given a convergent sequence $(\lambda\expn)_{n \geq 1}$ of integer partitions with limiting parameter $\omega \in \Omega \setminus \{\omega_1,\omega_{-1}\}$, the sequence of random variables $(n^{-2}(\majTn-\esper[\majTn]))_{n \geq 1}$ satisfies a large deviation principle on $\R$ with speed $n$ and good rate function $\Lambda_\omega^*$ (see \cite[Section 1.2]{DZ98} for the general definition of a large deviation principle). However, let us remark that we are also able to describe the second order of the probabilities of large deviations: it is described by the same residue $\Psi_\omega$ as in the asymptotic expansion of the log-Laplace transforms. In the literature, these results are usually called \emph{precise}, \emph{sharp} or \emph{strong} large deviation estimates.

\begin{remark}
As explained in Subsection \ref{sub:schur_functions}, the previous results regard for the major index of a  permutation $\sigma\expn$ taken uniformly at random in a RSK class with growing shape $\lambda\expn$. Let us then relate our asymptotic estimates to the case where $\sigma\expn$ is taken uniformly at random in the whole symmetric group $\sym(n)$. This amounts to first choose at random an integer partition $\lambda\expn$ according to the Plancherel measure
$$\proba\expn[\lambda\expn] = \frac{|\ST(\lambda\expn)|^2}{n!},$$
and then to take $T\expn \sim \mathcal{U}(\ST(\lambda\expn))$. The asymptotics of the Plancherel measure on integer partitions are nowadays well known; see for instance \cite{LS77,KV77,KV81,BDJ99,BOO00,Oko00,Oko01,IO02}. If one rescales the Young diagram $\lambda\expn \sim \proba\expn$ in both directions by a factor $n^{-\frac{1}{2}}$, then this renormalised shape converges towards a continuous limiting curve (Logan--Shepp--Kerov--Vershik law of large numbers). In particular, the Frobenius coordinates of $\lambda\expn$ all converge to $0$, so in the Thoma simplex $\Omega$, we have the convergence in probability 
$$\lambda\expn \to \omega_0=((0,0,\ldots),(0,0,\ldots)).$$ 
Therefore, one can expect that at least for the first order, the asymptotics of $\mathrm{maj}(\sigma\expn)$ with $\sigma\expn \sim \mathcal{U}(\sym(n))$ are described by Theorems \ref{main:log_laplace} and \ref{main:large_deviations}, but in the special case where $\omega=\omega_0$ and $\mu_{\omega} = \delta_0$ is the Dirac measure at $0$. Thus, one can predict:
\begin{align*}
\log \esper\!\left[\E^{z\left(\frac{\mathrm{maj}(\sigma\expn) - \esper[\mathrm{maj}(\sigma\expn)]}{n}\right)}\right] &= n\,\Lambda_{\omega_0}(z) + O(1) = n\,\int_{t=0}^1 \phi(tz)\DD{t} + O(1).
\end{align*}
This is indeed the case, as can be seen by using the following combinatorial argument. The distribution of the statistics $\mathrm{maj}$ over the symmetric group $\sym(n)$ is well known to be the same as the distribution of the number of inversions $\mathrm{inv}$; see \cite{Foa68,FS78} for a bijective proof of this identity. By induction on $n$, it is easy to see that the generating function of the number of inversions is:
$$\esper[\E^{z\,\mathrm{maj}(\sigma\expn)}]=\esper[\E^{z\,\mathrm{inv}(\sigma\expn)}] = \prod_{k=1}^n \frac{1+\E^z+\E^{2z}+\cdots + \E^{(k-1)z}}{k} = \prod_{k=1}^n \frac{\E^{kz}-1}{k(\E^z-1)}.$$
Using also the identity $\esper[\mathrm{maj}(\sigma\expn)] = \frac{n(n-1)}{4}$, we therefore get:
\begin{align*}
\log \esper\!\left[\E^{z\left(\frac{\mathrm{maj}(\sigma\expn) - \esper[\mathrm{maj}(\sigma\expn)]}{n}\right)}\right] &= \sum_{k=1}^n\varphi\!\left(\frac{kz}{n}\right) - n \,\varphi\!\left(\frac{z}{n}\right) - \frac{(n-1)z}{4} \\ 
&= \sum_{k=1}^n\phi\!\left(\frac{kz}{n}\right) - n \,\varphi\!\left(\frac{z}{n}\right) + \frac{z}{2} \\ 
&= n\int_{t=0}^1 \phi(tz)\DD{t} + \frac{1}{2}\,\phi(z) + O(n^{-1}).
\end{align*}
Here we used the Euler--Maclaurin formula in order to transform a Riemann sum into an integral; this argument will be used several times throughout the paper, see the proof of Lemma \ref{lem:control_gammaz} for the details. We thus recover the expected term of order $O(n)$. The term of order $O(1)$ does not correspond with $$\Psi_{\omega_0}(z) = \frac{1}{2}\left(\phi(z)-\frac{z^2}{12}\right),$$ and this is due to the fact that we are not taking into account the fluctuations of the RSK shape of $\sigma\expn$ around its limit $\omega_0$ (these fluctuations do not cancel when taking the log-Laplace transform, since it is a non-linear functional of probability measures). However, notice that the computation above yields an explicit term $\Psi(z)$ with order $O(1)$. So, one can adapt Theorem \ref{main:large_deviations} and its proof to the case of $\mathrm{maj}(\sigma\expn)$ with $\sigma\expn \sim \mathcal{U}(\sym(n))$, just by taking this residue $\Psi(z)= \frac{1}{2}\phi(z) $ instead of a function $\Psi_\omega(z)$. So, one obtains the strong large deviations of $\mathrm{maj}(\sigma\expn)$, with an estimate of the probabilities of large deviations instead of their logarithms. To our knowledge, this result is new.
\end{remark}

\begin{main}[Kolmogorov distance between the major index and its normal approximation]\label{main:berry_esseen}
Let $(\lambda\expn)_{n \geq 1}$ be a growing sequence of integer partitions; we suppose to simplify that
$$\max \left(\frac{\lambda_1\expn}{n}, \frac{\lambda_1^{(n)'}}{n}\right) \leq \frac{1}{2}$$
for $n\geq n_0\geq 4$ (in other words, the first row and the first column of $\lambda\expn$ are not too close to $n$). Then, for $n \geq n_0$,
$$\dkol\!\left(\frac{\majTn - \esper[\majTn]}{\sqrt{\var(\majTn)}},\,\mathcal{N}(0,1)\right)  \leq \frac{C}{\sqrt{n}} $$
for some universal constant $C\leq 30$.
\end{main}

Although Theorems \ref{main:large_deviations} and \ref{main:berry_esseen} are not directly related, a control on the Kolmogorov distance between an exponentially tilted version of $\majTn$ and its normal approximation will be an important step towards the proof of the strong large deviation estimates. We close our introduction by making several remarks, and by giving a short outline of the proofs our our main results.

\begin{remark}
 Taking the exponential of the asymptotic expansion given by Theorem \ref{main:log_laplace}, we get:
 $$\esper\!\left[\E^{z(X\expn - \esper[X\expn])}\right] = \E^{n\,\Lambda_\omega(z)}\, \E^{\Psi_{\omega}(z)}\,(1+o(1)).$$
 Informally, this can be understood as follows: as $n$ goes to infinity, the centered random variable $X\expn - \esper[X\expn]$ is approximated by the sum of $n$ independent and identically distributed random variables with Laplace transform $\E^{\Lambda_\omega(z)}$, plus some remainder which is encoded in the Laplace sense by the multiplicative residue $\E^{\Psi_\omega(z)}$. This viewpoint is the one of \emph{mod-$\phi$ convergent sequences of random variables}, and this framework has been explored in particular in \cite{JKN11,KN12,DKN15,FMN16,FMN19,CDMN20}. An important assumption in these works is that $\E^{\Lambda_\omega(z)}$ is the Laplace transform of an infinitely divisible distribution, for instance a Poisson or Gaussian distribution. For the major index of a random standard tableau, the same ideas apply, but $\E^{\Lambda_\omega(z)}$ is not anymore the Laplace transform of a probability distribution (see Remark \ref{rem:not_log_laplace}). Thus, we obtain our strong large deviation results by comparing in the Laplace sense a random variable $X\expn$ and something which is not a random variable, but which plays an analogous role in the computations. We believe that this idea is of independent interest, and that it could lead to important extensions of the aforementioned framework of mod-$\phi$ convergent sequences.
\end{remark}

\begin{remark}
Theorem \ref{thm:BKS} relies on the analysis of Equation \eqref{eq:cumulant_maj_ST}: indeed, this formula implies that the cumulants of order $r\geq 3$ of the rescaled random major indices go to $0$ as $n$ goes to infinity, as long as the first row or first column of $\lambda=\lambda\expn$ does not contain almost all the cells. On the other hand, the theory of mod-Gaussian sequences admits a formulation in terms of upper bounds on the cumulants, see \cite[Section 9]{FMN16} and \cite[Sections 4-5]{FMN19}. So, it is natural to find the usual results from the theory of mod-$\phi$ sequences (Berry--Esseen estimates and strong large deviations) in the setting of random standard tableaux and their major indices. This also explains why we shall start our analysis by looking in details at the cumulants $\kappa^{(r)}(\majTn)$; although one could work directly with the log-Laplace transforms, it is much easier to first analyse the coefficients of these generating series, and then to resum them; see Remark  \ref{rem:why_cumulant}.
\end{remark}

\begin{remark}
In \cite{BS20}, several extensions of Theorem \ref{thm:BKS} are proved for other statistics of other random combinatorial objects, and in particular for the rank of a random semistandard tableau (semistandard means that the rows of the tableau are only weakly increasing, and therefore that one allows repetitions in the entries of the tableau). As the formula for the generating function of this statistics has a form very similar to the form of the generating series of the major index of a random standard tableau, it is almost certain that our main results have direct analogues in this setting, and with similar proofs. 
\end{remark}

The steps of the proof of our main results are the following:
\begin{itemize}
    \item We relate the cumulants of $\majT$ to several other observables of the integer partition $\lambda$, in particular the Frobenius moments (Subsection \ref{sub:descents_and_frobenius_moments}) and the power sums of the contents (Subsection \ref{sub:symmetric_functions_of_the_contents}). This enables one to prove that $\kappa^{(r)}(\majT)$ is an observable of $\lambda$ in the sense of Kerov and Olshanski, and to determine the asymptotic behavior of each cumulant, thereby proving Theorem \ref{thm:asymptotics_cumulants} (Subsection \ref{sub:asymptotics_of_the_cumulants}).
    \item Theorem \ref{main:berry_esseen} is then an easy consequence of the asymptotics of cumulants and of a general result relating the growth of the cumulants to the distance to the Gaussian distribution \cite[Corollary 30]{FMN19}; see Subsection \ref{sub:proof_B}.
    \item We resum the limits of the cumulants, and we prove that the limiting functions are the ones that appear in Theorem \ref{main:log_laplace} (Subsection \ref{sub:proof_A}). We also show that the asymptotic estimates of the log-Laplace transforms hold not only on a disc of convergence of the power series $\phi(z)$ and $\varphi(z)$, but in fact on the domain $\frac{1}{2}\,\mathscr{D}_0$.
    \item Finally, in Subsection \ref{sub:proof_C}, given a convergent sequence $(\lambda\expn)_{n \geq 1}$ with limiting parameter $\omega \in \Omega$, we study the behavior of the false Laplace transform $\E^{\Lambda_\omega(z)}$ on a line $z = h+\I \xi$ with $h \neq 0$ fixed and $\xi \in \R$. We choose the parameter $h$ so that the random variable $\tildeX\expn$ obtained from $X\expn$ by an exponential tilting with strength $h$ has its mean close to $\esper[X\expn]+ny$ instead of $\esper[X\expn]$, and we prove a central limit theorem with explicit Berry--Esseen estimate for the tilted random variable $\tildeX\expn$. This central limit theorem leads to our Theorem \ref{main:large_deviations} by standard arguments from the theory of large deviations.
\end{itemize}

\bigskip

\section{Observables of integer partitions}\label{sec:observables}
In this section, $\lambda=\lambda\expn$ is an integer partition of size $n$, and we aim to prove Theorem \ref{thm:asymptotics_cumulants} by rewriting the equation for cumulants \eqref{eq:cumulant_maj_ST} in the Kerov--Olshanski algebra $\obs$ of observables of integer partitions. We refer to \cite{KO94,IO02} and \cite[Chapter 7]{Mel17} for details on this algebra, which plays a major role in the asymptotic representation theory of symmetric groups and in the study of related models of large (random) integer partitions; see also \cite{OO98} for the connection with shifted symmetric functions, and \cite{IK99} for the realisation of $\obs$ as a subalgebra of the algebra of partial permutations. We shall mostly use the basis of Frobenius moments $(p_\mu^{\mathrm{F}})_{\mu \in \ym}$ of $\obs$, but in order to deal with the power sums of contents $(p_\mu^\oblong)_{\mu \in \ym}$, it will be convenient to also work with the canonical basis of \emph{renormalised character values} $(\varSigma_\mu)_{\mu \in \ym}$ (see Subsection \ref{sub:symmetric_functions_of_the_contents}).
\medskip

\subsection{Cumulants and the Kerov--Olshanski algebra}
We denote $\obs$ the algebra of functions from $\ym = \bigsqcup_{n \in \N} \ym(n)$ to $\R$ which is spanned algebraically by the Frobenius moments $p_k^{\mathrm{F}}$ with $k \geq 1$. By \cite[Proposition 1.5]{IO02}, the Frobenius moments are algebraically independent, so if 
$$p_\mu^{\mathrm{F}} = p_{\mu_1}^{\mathrm{F}}\,p_{\mu_2}^{\mathrm{F}} \cdots p_{\mu_\ell}^{\mathrm{F}}$$
for any integer partition $\mu = (\mu_1,\mu_2,\ldots,\mu_{\ell}) \in \ym$, then $(p_\mu^{\mathrm{F}})_{\mu \in \ym}$ forms a linear basis of $\obs$. In the following, the elements of $\obs$ will be called \emph{observables} of Young diagrams or integer partitions. We endow $\obs$ with the gradation $\deg p_\mu^{\mathrm{F}} = |\mu|$; then, $\obs$ becomes a graded algebra, which is isomorphic to the graded algebra $\Sym$ of symmetric functions (see \cite[Chapter I]{Mac95}): a natural isomorphism consists in sending each Frobenius moment $p_\mu^{\mathrm{F}}$ to the corresponding product of Newton power sums $p_\mu$. Theorem \ref{thm:asymptotics_cumulants} admits then the following reformulation:

\begin{theorem}[The cumulants belong to the algebra of observables]\label{thm:cumulants_are_observables}
For $r \geq 2$, the $r$-th cumulant of the statistics $\majT$ with $T$ uniformly chosen in $\ST(\lambda)$ is an element of the algebra of observables $\obs$. Moreover, $\deg \kappa^{(r)}(\majT) = r+1$, and the terms with degree $r+1$ and $r$ in $\kappa^{(r)}(\majT)$ are given by the formula of Theorem \ref{thm:asymptotics_cumulants}.
\end{theorem}

\noindent Note that Theorem \ref{thm:cumulants_are_observables} immediately implies the uniform remainder $O(n^{r-1})$ in Theorem \ref{thm:asymptotics_cumulants}. Indeed, for any $\lambda\expn \in \ym(n)$, $p_\mu^{\mathrm{F}}(\lambda\expn) = n^{|\mu|}\,p_{\mu}(\omega_{\lambda\expn})$ for any $\mu \in \ym$, and on the other hand, $|p_\mu(\omega)| \leq 1$ for any $\omega \in \Omega$ and any $\mu \in \ym$. So, if the terms with degree smaller than $r-1$ in $\kappa^{(r)}(\majT)$ write explicitly as 
$$\sum_{|\nu| \leq r-1} c_\nu\,p_\nu^{\mathrm{F}},$$
then the evaluation of this remaining observable on an integer partition $\lambda\expn \in \ym(n)$ is bounded from above by $$\left(\sum_{|\nu| \leq r-1} |c_\nu|\right)\,n^{r-1},$$ 
so it is a uniform $O(n^{r-1})$, regardless of the choice of an integer partition $\lambda\expn$.
\bigskip

The starting argument of the proof of Theorem \ref{thm:cumulants_are_observables} is the following:
\begin{lemma}\label{lem:subtle_geometry}
For any integer partition $\lambda$ with length smaller than $n$, we have the equality of multisets:
\begin{align*}
&\{h(\oblong),\,\,\oblong \in \lambda \}\sqcup \{\lambda_i-\lambda_j+j-i,\,\,1\leq i<j\leq n\} \\ 
&=\{n+c(\oblong),\,\,\oblong \in \lambda\}\sqcup \{1^{(n-1)},2^{(n-2)},\ldots,(n-1)\}.
\end{align*}
where $c(\oblong)=j-i$ is in $i$-th row and the $j$-th column of the Young diagram $\lambda$ (\emph{content} of a cell), and $a^{(b)}$ denotes the sequence with $b$ values equal to $a$.
\end{lemma}

\begin{proof}
This identity is one of the main argument of the proof of the Stanley hook length formula; see for instance \cite[Proposition 4.63]{Mel17}.
\end{proof}

As a consequence, for any $r$ positive integer and any integer partition $\lambda$ with size $n$, we can rewrite
\begin{align*}
\sum_{\oblong \in \lambda} (h(\oblong))^r = \sum_{\oblong \in \lambda} (n+c(\oblong))^r + \sum_{i=1}^{n} (n-i)\,i^{r} - \sum_{1\leq i<j \leq n} (\lambda_i^* - \lambda_j^*)^r
\end{align*}
where $\lambda_i^*=\lambda_i-i+\frac{1}{2}$ is the $i$-th \emph{descent} of $\lambda$. So, Equation \eqref{eq:cumulant_maj_ST} becomes
\begin{equation}
\kappa^{(r)}(\majT - b(\lambda)) = \frac{B_r}{r}\,\left(\sum_{1\leq i<j \leq n} (\lambda_i^*-\lambda_j^*)^r - \sum_{\oblong \in \lambda} (n+c(\oblong))^r - \sum_{i=1}^{n} (n-i-1)\,i^{r}\right).\label{eq:cumulant_maj_ST3}
\end{equation}
Let us also write the consequence of Lemma \ref{lem:subtle_geometry} for the log-Laplace transform of $\majT$, without extraction of its coefficients.
\begin{proposition}
For any integer partition $\lambda$ with size $n \geq 1$ and any $z $ such that $2z \in \mathscr{D}_0$,
\begin{align*}
&\log \esper\!\left[\E^{z\,\frac{\majT}{n}}\right] \\ 
&= \frac{b(\lambda)\,z}{n} + \sum_{1\leq i<j\leq n}\varphi\!\left(\frac{(\lambda_i^*-\lambda_j^*)z}{n}\right) - \sum_{\oblong \in \lambda} \varphi\!\left(\left(1+\frac{c(\oblong)}{n}\right)z\right)- \sum_{k=1}^n (n-k-1)\,\varphi\!\left(\frac{kz}{n}\right),
\end{align*}
where $T \sim \mathcal{U}(\ST(\lambda))$.
\end{proposition}

\begin{proof}
The assumption $2z \in \mathscr{D}_0$ ensures that both sides of the formula are well defined, since $\lambda_i^*-\lambda_j^* \leq 2n$ for any pair $(i,j)$, and $n+c(\oblong) \leq 2n$ for any cell $\oblong \in \lambda$. Set $y = \frac{z}{n}$. We have
\begin{align*}
\esper[\E^{y\,\majT}] &= \frac{\E^{b(\lambda)\,y}}{|\ST(\lambda)|}\,\frac{\prod_{k=1}^n (\E^{ky}-1)}{\prod_{\oblong \in \lambda}(\E^{h(\oblong)\,y}-1)} = \exp\left(b(\lambda)\,y+ \sum_{k=1}^n \varphi(ky) - \sum_{\oblong \in \lambda} \varphi(h(\oblong)y)\right)
\end{align*}
by using the hook length formula $|\ST(\lambda)| = \frac{\prod_{k=1}^n k}{\prod_{\oblong \in \lambda}h(\oblong)}$. As before, the result follows by replacing the sum over hook lengths by three sums involving the integers $k\in \lle 1,n\rre$, the contents of the cells of $\lambda$ and the descent coordinates.
\end{proof}
\medskip

In order to prove Theorem \ref{thm:cumulants_are_observables}, it suffices to treat the case of even cumulants, as the odd cumulants vanish. The cumulants of order $r \geq 2$ are invariant by translation, so $\kappa^{(r)}(\majT - b(\lambda)) = \kappa^{(r)}(\majT)$. For any even integer $r \geq 2$, set 
\begin{align*}
\alpha_r &= \frac{1}{2}\,\sum_{1\leq i,j\leq n} (\lambda_i^* - \lambda_j^*)^r \quad;\quad \beta_r = \sum_{\oblong \in \lambda} (n+c(\oblong))^r \quad;\quad \gamma_r = \sum_{i=1}^n (n-i-1)\,i^r.
\end{align*}
The even cumulant $\kappa^{(r)}(\majT)$ is proportional to $\alpha_r - \beta_r - \gamma_r$, and we shall analyse separately these quantities. We shall also work with the following holomorphic functions on $\frac{1}{2}\mathscr{D}_0$:
\begin{align*}
\alpha(z) &= \frac{b(\lambda)\,z}{n}  + \sum_{1\leq i<j\leq n}\varphi\!\left(\frac{(\lambda_i^*-\lambda_j^*)}{n}\,z\right) ;\\ 
\beta(z) &= \sum_{\oblong \in \lambda} \varphi\!\left(\left(1+\frac{c(\oblong)}{n}\right)z\right);\\
\gamma(z) &= \sum_{k=1}^n (n-k-1)\,\varphi\!\left(\frac{kz}{n}\right).
\end{align*}
We have $\log \esper[\E^{z\,\frac{\majT}{n}}] = \alpha(z)-\beta(z)-\gamma(z)$, and for $r \geq 2$, the $r$-th coefficient of the Taylor expansion of $\alpha(z)$ at $z=0$ is equal to $\frac{B_r}{r}\,\frac{\alpha_r}{n^r}$, and similarly for the functions $\beta(z)$ and $\gamma(z)$.\medskip

 Note that the term $\gamma_r$ is easy to compute: we have
$$
\delta_r = \sum_{i=1}^n i^r = \frac{n^{r+1}}{r+1} + \sum_{s=1}^r \frac{B_s}{s!}\,r^{\downarrow s-1}\,n^{r+1-s} = \frac{(p_1^{\mathrm{F}})^{r+1}}{r+1} + \sum_{s=1}^r \frac{B_s}{s!}\,r^{\downarrow s-1}\,(p_1^{\mathrm{F}})^{r+1-s}
$$
for $r \geq 1$, and $\gamma_r = (n-1)\,\delta_r - \delta_{r+1} = (p_1^{\mathrm{F}} - 1)\,\delta_r - \delta_{r+1}$. Therefore, $\gamma_r$ is an observable with degree $r+2$, and we have the following expansion with respect to the degree for $r \geq 2$:
\begin{equation}
\gamma_r =\frac{(p_1^{\mathrm{F}})^{r+2}}{(r+1)(r+2)}  - \frac{(p_1^{\mathrm{F}})^{r+1}}{r+1}  -\frac{7\,(p_1^{\mathrm{F}})^r}{12} + \text{terms of degree smaller than }r-1.\label{eq:gamma_r}\tag{$\gamma$}
\end{equation}
In order to prove Theorem \ref{main:log_laplace}, we shall also use the following control of the function $\gamma(z)$.

\begin{lemma}\label{lem:control_gammaz}
If $z \in \mathscr{D}_0$, then
$$\left|\gamma(z) - \int_{t=0}^1 (n^2(1-t)-n)\, \varphi(tz)\DD{t} \right|$$
is locally uniformly bounded.
\end{lemma}

\begin{proof}
The Euler--Maclaurin formula ensures that for $f$ smooth function on $[0,n]$,
$$\sum_{k=1}^n f(k) = \int_0^n f(x)\DD{x} + \frac{f(n)-f(0)}{2} + \frac{f'(n)-f'(0)}{12} + O(n\|f''\|_\infty),$$
with a universal constant in the $O(\cdot)$. 
With $f(x) = (n-x-1)\,\varphi(\frac{xz}{n})$, $\|f''\|_\infty$ is a $O(n^{-1})$, with a uniform constant in the $O(\cdot)$ if $z$ stays in a compact subset $K$ of $\mathscr{D}_0$. Indeed, if $x \in [-1,1]$, then the scaled parameter $xz$  also stays in a compact subset $K'$ of $\mathscr{D}_0$; see Figure \ref{fig:scale_compact}, where $K$ is in purple and $K'\setminus K$ is in blue. 
\begin{figure}[ht]

\begin{center}
\begin{tikzpicture}[scale=0.75]
\draw [->] (-5,0) -- (5,0);
\draw [->] (0,-5) -- (0,5);
\draw [very thick,red] (0,3.14) -- (0,4.9);
\draw [very thick,red] (0,-3.14) -- (0,-5);
\draw [very thick,red] (-0.1,3.14) -- (0.1,3.14);
\draw [very thick,red] (-0.1,-3.14) -- (0.1,-3.14);
\draw [red] (-0.6,3.14) node {$2\I \pi$};
\draw [red] (0.75,-3.14) node {$-2\I \pi$};
\fill [shift={(2,2)},blue!50!white] (-75:1.5) -- (-2,-2) -- (165:1.5);
\draw [shift={(2,2)},thick,blue] (-75:1.5) -- (-2,-2) -- (165:1.5);
\fill [shift={(-2,-2)},blue!50!white] (105:1.5) -- (2,2) -- (-15:1.5);
\fill [blue!50!white] (-2,-2) circle (1.5);
\draw [shift={(-2,-2)},thick,blue] (105:1.5) -- (2,2) -- (-15:1.5);
\draw [shift={(-2,-2)},thick,blue] (105:1.5) arc (105:345:1.5);
\fill [violet!30!white] (2,2) circle (1.5);
\draw [thick,violet] (2,2) circle (1.5);
\draw [violet] (3,2) node {$K$};
\end{tikzpicture}
\end{center}
\caption{If $K$ is a compact subset of $\mathscr{D}_0$, then $[-1,1]K = K'$ is also a compact subset of $\mathscr{D}_0$.\label{fig:scale_compact}}
\end{figure}
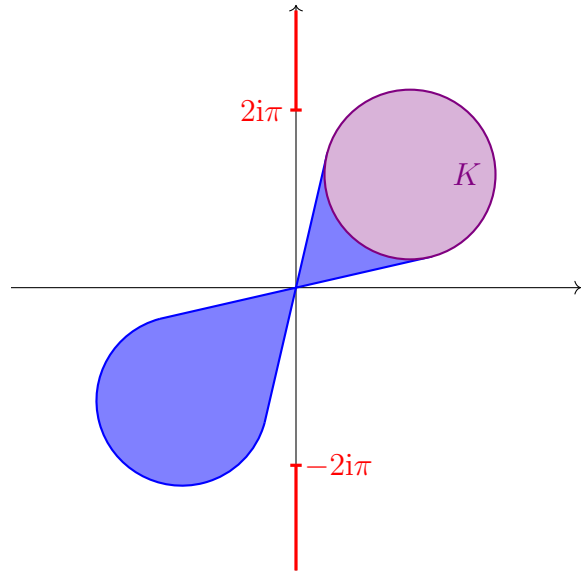
 The result follows now immediately by making the change of variable $x=nt$ in the main term of the Euler--Maclaurin formula.
\end{proof}
\medskip

\subsection{Descents and Frobenius moments}\label{sub:descents_and_frobenius_moments}
Given an integer partition $\lambda$, let us denote $A(\lambda)=(a_1,\ldots,a_d)$ and $B(\lambda)=(-b_1,\ldots,-b_d)$, so that $p_r^{\mathrm{F}}(\lambda) = \sum_{x \in A(\lambda)} x^r - \sum_{x \in B(\lambda)} x^r$. The signed Frobenius coordinates are related to the set of \emph{descent coordinates} $D(\lambda)= \{\lambda_i-i+\frac{1}{2},\,i \geq 1\}=\{\lambda_i^*,\,i\geq 1\}$ by the relations:
$$A(\lambda) = \Z'_+ \cap D(\lambda)\qquad;\qquad B(\lambda) = \Z'_- \setminus (\Z_-' \cap D(\lambda)),$$
where $\Z' = \Z + \frac{1}{2}$ is the set of half-integers; see \cite[Proposition 1.1]{IO02}. These relations are obvious if one draws the Young diagram $\lambda$ with the Russian convention, that is to say rotated by $45$ degrees and with cells of area $2$: see Figure \ref{fig:descent_coordinates}.
\begin{figure}[ht]
\begin{center}
\begin{tikzpicture}[scale=0.75]
\draw [->] (-6,0) -- (6,0);
\draw [->] (0,-0.3) -- (0,6);
\draw [thick] (0,0) -- (5,5) -- (4,6) -- (3,5) -- (2,6) -- (0,4) -- (-1,5) -- (-3,3) -- (0,0);
\draw [thick] (4,4) -- (3,5) -- (-1,1);
\draw [thick] (3,3) -- (1,5);
\draw [thick] (2,2) -- (0,4) -- (-2,2);
\draw [thick] (1,1) -- (-2,4);
\draw [thick] (6,6) -- (5,5);
\draw [thick] (-3,3) -- (-6,6);
\draw [thick, red] (4.5,5.5) -- (4.5,0);
\draw [thick, red] (2.5,5.5) -- (2.5,0);
\draw [thick, red] (-0.5,4.5) -- (-0.5,0);
\draw [thick, red] (-3.5,3.5) -- (-3.5,0);
\draw [thick, red] (-4.5,4.5) -- (-4.5,0);
\draw [thick, red] (-5.5,5.5) -- (-5.5,0);
\draw [thick,red,shift={(0.1,0.1)}] (5,5) -- (4,6);
\draw [thick,red,shift={(0.1,0.1)}] (3,5) -- (2,6);
\draw [thick,red,shift={(0.1,0.1)}] (-1,5) -- (0,4);
\draw [thick,red,shift={(0.1,0.1)}] (-6,6) -- (-3,3);
\foreach \x in {-5.5,-4.5,-3.5,-2.5,-1.5,-0.5,0.5,1.5,2.5,3.5,4.5}
\draw (\x,-0.2) -- (\x,0.2);
\foreach \x in {-5.5,-4.5,-3.5,-0.5,2.5,4.5}
\fill [red]  (\x,0) circle (0.13);
\end{tikzpicture}
\end{center}
\caption{The descent coordinates of the integer partition $\lambda=(5,4,2)$.\label{fig:descent_coordinates}}
\end{figure}
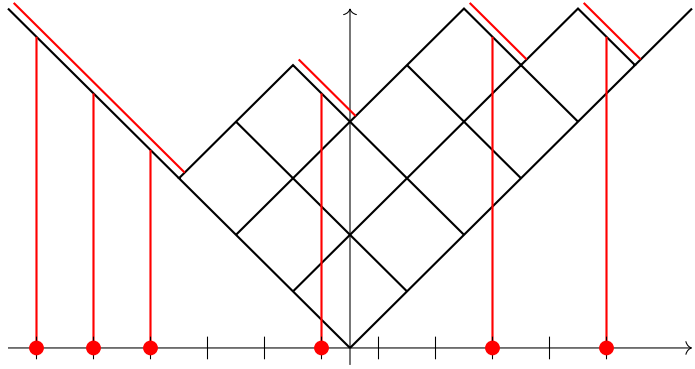

\noindent As a consequence,
$$p_{r}^{\mathrm{F}}(\lambda) = \sum_{i \geq 1} \left(\lambda_i-i+\frac{1}{2}\right)^r - \left(-i+\frac{1}{2}\right)^r;$$
see \cite[Proposition 1.4]{IO02}. Therefore, for $r\geq 2$ even, the term $\alpha_r$ admits the following expansion in terms of the Frobenius moments and of the sums $\delta_r^* = \sum_{i=1}^n (i-\frac{1}{2})^r$:
\begin{align*}
\alpha_r &= \frac{1}{2}\,\sum_{s=0}^r \binom{r}{s}\,(-1)^{r-s} \left(\sum_{i=1}^n (\lambda_i^*)^s\right)\left(\sum_{j=1}^n(\lambda_j^*)^{r-s}\right) \\
&= \frac{1}{2}\,\sum_{s=0}^r \binom{r}{s}\,(-1)^{r-s} (p_s^\mathrm{F} +(-1)^s \delta_s^*)(p_{r-s}^{\mathrm{F}} + (-1)^{r-s} \delta_{r-s}^*)\\ 
 &=\frac{1}{2}\,\sum_{s=0}^r \binom{r}{s}\, ((-1)^{s}\, p_s^\mathrm{F} \,p_{r-s}^\mathrm{F}  +(-1)^s\, \delta_s^*\,\delta_{r-s}^* + 2\,\delta_s^* \,p_{r-s}^\mathrm{F} ).
\end{align*}
Obviously, $\delta_r^* = \sum_{s=0}^r \binom{r}{s} (-\frac{1}{2})^{r-s}\, \delta_s$, so the formula above proves that $\alpha_r$ belongs to the algebra $\obs$. Moreover, it is of degree $r+2$, and we obtain its leading terms:
\begin{equation}
\alpha_r = \frac{(p_1^{\mathrm{F}})^{r+2}}{(r+1)(r+2)} + \frac{1}{2}\,\sum_{s=0}^r \binom{r}{s}\, \left(\frac{2 (p_1^{\mathrm{F}})^{s+1}\,p_{r-s}^{\mathrm{F}}}{s+1} + (-1)^{s}\, p_s^\mathrm{F} \,p_{r-s}^\mathrm{F} \right) - \frac{(p_1^{\mathrm{F}})^r}{12}+\cdots\label{eq:alpha_r}\tag{$\alpha$}
\end{equation}
by computing some combinatorial sums, for instance $\sum_{s=0}^r \binom{r}{s}\,\frac{(-1)^s}{(s+1)(r-s+1)} = \frac{2}{(r+1)(r+2)}$ for $r$ even. The remainder is an observable with degree smaller than $r-1$, for any $r \geq 2$. The following lemma is the corresponding control of the function $\alpha(z)$: 

\begin{lemma}
\label{lem:control_alphaz}
If $z \in \frac{1}{2}\mathscr{D}_0$, then
$$\left|\alpha(z) - \int_{t=0}^1 \int_{x=-1}^1 \left(n^2(1-t)\,\varphi(tz) + n\,\frac{\varphi((x+t)z)-\varphi(tz)}{x}\right)\mu\expn(\!\DD{x})\DD{t} \right|,$$
is locally uniformly bounded, independently from the growing sequence $(\lambda\expn)_{n\geq 1}$.
\end{lemma}

\begin{proof}
First, let us remove the linear terms from $\alpha$ and $\varphi$. We have
\begin{align*}
b(\lambda)+ \frac{1}{2}\sum_{1\leq i<j\leq n}(\lambda_i^* - \lambda_j^*) &= \frac{1}{2}\left(\sum_{i=1}^n (n-1)\lambda_i + \sum_{1\leq i<j \leq n} (j-i) \right) \\ 
&= \frac{n(n-1)}{2} + \frac{n(n-1)(n+1)}{12} = \frac{n^3}{12}+\frac{n^2}{2} - \frac{7n}{12},
\end{align*}
and on the other hand,
$$B_1\int_{t=0}^1 \int_{x=-1}^1 \left(n^2(1-t)\,tz + n\,z\right)\mu\expn(\!\DD{x})\DD{t} = \frac{n^2z}{12}+\frac{nz}{2},$$
so the lemma is equivalent to the locally uniform boundedness of 
$$\overline{\alpha}(z) - \int_{t=0}^1 \int_{x=-1}^1 \left(n^2(1-t)\,\phi(tz) + n\,\frac{\phi((x+t)z)-\phi(tz)}{x}\right)\mu\expn(\!\DD{x})\DD{t} $$
with $\overline{\alpha}(z) = \frac{1}{2}\sum_{1\leq i,j\leq n}\phi(\frac{(\lambda_i^* - \lambda_j^*)z}{n})$. We now consider the interval of half-integers $$I_n=\left\{-n+\frac{1}{2},-n+\frac{3}{2},\ldots,-\frac{1}{2}\right\},$$ which contains $n$ entries. If $F(\lambda) = A(\lambda) \sqcup B(\lambda)$ and $D_n = \{\lambda_i-i+\frac{1}{2},\,\,1\leq i \leq n\}$, then by the discussion at the beginning of this paragraph, $D_n$ is the symmetric difference $(F(\lambda) )\Delta (I_n)$. Therefore, for any even function $G$,
\begin{align}
\sum_{x,y \in D_n} G(y-x) &= \sum_{x,y} 1_{x \in D_n} 1_{y \in D_n}\,G(y-x) \notag\\ 
&= \sum_{x,y} (1_{x \in F} + 1_{x\in I_n} -2\, 1_{x \in B}) (1_{y \in F} +1_{y \in I_n} - 2\,1_{y \in B})\,G(y-x) \notag\\ 
&= \sum_{x,y} (1_{x \in A} + 1_{x\in I_n} -1_{x \in B}) (1_{y \in A} +1_{y \in I_n} - 1_{y \in B})\,G(y-x) \notag\\ 
&=\sum_{x,y \in F} \mathrm{sgn}(xy)\,G(y-x) + 2\sum_{x\in F,\,y \in I_n} \mathrm{sgn}(x)\, G(y-x) +\sum_{x,y \in I_n} G(y-x).\label{eq:inclusion_exclusion}
\end{align}
With $G(t) = \frac{1}{2}\,\phi(\frac{tz}{n})$, the first term of \eqref{eq:inclusion_exclusion} rewrites as 
\begin{align*}
\sum_{x,y \in F} \mathrm{sgn}(xy)\,G(y-x) &= \frac{1}{2} \int_{x=-1}^1 \int_{y=-1}^1 \frac{\phi((y-x)z)}{xy}\, \mu^{(n),\otimes 2}(\!\DD{x} \DD{y}) \\ 
&= \frac{1}{2} \int_{x=-1}^1 \int_{y=-1}^1 \frac{\phi((y-x)z)-\phi(yz) - \phi(-xz)}{xy}\, \mu^{(n),\otimes 2}(\!\DD{x} \DD{y}) ,
\end{align*}
by using on the second line the vanishing of $\int_{-1}^1 \frac{1}{x}\,\mu\expn(\!\DD{x})$, which comes from the fact that $\lambda$ has the same number of positive and negative Frobenius coordinates. The double integral above is involved in the $O(1)$ term of the asymptotics of the log-Laplace transform; let us just explain why it is locally uniformly bounded with respect to the variable $z$. We are integrating against a probability measure on $[-1,1]^2$ the function
$$G(x,y,z)= \begin{cases}
    \frac{\phi(z(y-x))-\phi(zy)-\phi(-zx)}{xy} &\text{if }xy\neq 0,\\
    \frac{1}{y^2} \left(1-\frac{zy}{2}\coth \frac{zy}{2}\right)&\text{if }x=0 \text{ and }y\neq 0,\\ 
    \frac{1}{x^2} \left(1-\frac{zx}{2}\coth \frac{zx}{2}\right)&\text{if }x\neq 0 \text{ and }y= 0,\\ 
    -\frac{z^2}{12}&\text{if }x=y=0,
\end{cases}$$ 
which is continuous on $[-1,1]^2 \times \mathscr{D}_0$. Therefore, if $z$ stays in a compact subset $K$ of $\mathscr{D}_0$, then $G(x,y,z)$ is uniformly bounded by some constant $C(K)$ (continuous function on a compact), and the integral $\iint_{[-1,1]^2} G(x,y,z)\,\mu^{(n),\otimes 2}(\!\DD{x} \DD{y})$ is also bounded by $C(K)$.
\medskip

We now deal with the two remaining terms of Equation \eqref{eq:inclusion_exclusion}. We have
\begin{align*}
2\sum_{x\in F,\,y \in I_n} \mathrm{sgn}(x)\, G(y-x) &= \int_{x=-1}^1 \left(\sum_{y \in -I_n}\frac{\phi((x+\frac{y}{n})z)}{x}\right) \,\mu\expn(\!\DD{x}) \\ 
&=n\,\int_{t=0}^1\int_{x=-1}^1 \frac{\phi((x+t)z)}{x}\,\mu\expn(\!\DD{x})\DD{t} + O(n^{-1}).
\end{align*}
Indeed, the sum over half-integers in $-I_n$ is a Riemann sum with middle points, and for any smooth function $f$ on $[0,1]$,
$$\left|\sum_{y \in -I_n} f\!\left(\frac{y}{n}\right) - n\int_{0}^1 f(t)\DD{t}\right| \leq \frac{\|f''\|_\infty}{24n^2}.$$
We use this estimate with $f(t) = \phi((x+t)z)$, and the remainder is a uniform $O(n^{-1})$ if $z$ stays in a compact subset of $\mathscr{D}_0$. By using as before the vanishing of $\int_{-1}^1 \frac{1}{x}\,\mu\expn(\!\DD{x})$, we can replace $\phi((x+t)z)$ by the difference $\phi((x+t)z)-\phi(tz)$, thereby getting the term of order $n$ in our asymptotic expansion of $\overline{\alpha}(z)$. Finally, 
$$\sum_{x,y \in I_n} G(y-x) = \frac{n^2}{2} \int_{u=0}^1 \int_{v=0}^1 \phi(|u-v|z) \DD{u}\DD{v} + O(n^{-1})$$
by using twice the argument of Riemann sums with middle points, and the even character of the function $\phi$. Since the image of the Lebesgue measure on the square $[0,1]^2$ by the map $(u,v) \mapsto |u-v|$ is the density $2(1-t)\DD{t}$ on $[0,1]$, we thus obtain the term of order $n^2$.
\end{proof}

\begin{remark}\label{rem:why_cumulant}
The integral formula from Lemma \ref{lem:control_alphaz} can be derived directly from Equation \eqref{eq:alpha_r} if $|z|<\pi$, by resummation of the series $\varphi$. The interest of Lemma \ref{lem:control_alphaz} lies in the control on the whole domain $\frac{1}{2}\mathscr{D}_0$; and the same remark applies to the functions $\beta(z)$ and $\gamma(z)$, see Lemmas  \ref{lem:control_gammaz} and \ref{lem:control_betaz}. Technically, we could have dealt only with the analytic functions $\alpha(z)$, $\beta(z)$ and $\gamma(z)$, but it is very difficult to guess the correct form of the integrals approximating these functions. For instance, the term of order $n$ in the control of $\alpha(z)$ can be written as
$$\int_{t=0}^1 \int_{x=-1}^1 \frac{\varphi((x+t)z)-\varphi(tz)}{x}\mu\expn(\!\DD{x})\DD{t}
$$
or as
$$\int_{t=0}^1 \int_{x=-1}^1 \frac{(1+x)\,\varphi((1+x)tz)-\varphi(tz)-x\,\varphi(txz)}{x}\mu\expn(\!\DD{x})\DD{t},
$$
and it is almost impossible to guess the correct form without using the cumulants as a guide. In order to prove that the integrals above are equal, it seems also necessary to expand the power series $\varphi$, which again leads us to the cumulants (in this regard, see the end of the proof of Lemma \ref{lem:control_betaz}). Another reason why cumulants are useful is that thanks to the properties of the Kerov--Olshanski algebra of observables, it is easy to compute the limits of the cumulants and then sum these limits; and in particular this is much easier than to find directly the $O(1)$ term in the asymptotic expansion of the log-Laplace transform of the scaled random major index. The validity of the summation of the limits will be explained in Subsection \ref{sub:proof_A}.
\end{remark}
\medskip

\subsection{Symmetric functions of the contents}\label{sub:symmetric_functions_of_the_contents}
For any $r \geq 1$, $\beta_r = \sum_{\oblong \in \lambda} (n+c(\oblong))^r$ is a symmetric function of the contents of the cells of the Young diagram of $\lambda$ with coefficients in $\Z[n]$, so by \cite[Theorem 8.26]{Mel17}, it is an observable in the Kerov--Olshanski algebra. In order to compute its leading terms, let us introduce another combinatorial basis of this algebra, which is based on the representation theory of the symmetric groups. For any integer partitions $\lambda$ and $\mu$ with respective sizes $n$ and $k$, denote $n^{\downarrow k} = n(n-1)\cdots (n-k+1)$, and set
$$\varSigma_\mu(\lambda) = \begin{cases}
    n^{\downarrow k}\,\frac{\ch^\lambda(\mu \sqcup 1^{(n-k)})}{\ch^\lambda(1\expn)} &\text{if }n \geq k,\\
    0 &\text{if }n<k.
\end{cases}$$
Here, $\ch^\lambda$ denotes the character of the irreducible representation of the symmetric group $\sym(n)$ with label $\lambda$; later we shall use the notation 
$$\chi^\lambda=\frac{\ch^\lambda(\cdot)}{\ch^\lambda(1\expn)}$$
 for the \emph{normalised} irreducible character. The irreducible character $\ch^\lambda$ can be seen as a function on the set $\ym(n)$ of integer partitions of size $n$ (the conjugacy classes of the symmetric group $\sym(n$), and also as a function from the group algebra $\C\sym(n)$ to $\C$. We refer to \cite{Sag01} and to \cite[Chapters 2 and 3]{Mel17} for details on the representation theory of $\sym(n)$ and the construction of its irreducible representations. It turns out that the family of symbols $(\varSigma_\mu)_{\mu \in \ym}$, viewed as functions from $\ym$ to $\R$, forms a linear basis of the algebra $\obs$. Moreover, $\deg \varSigma_\mu = |\mu|$ for any integer partition $\mu$, and $\varSigma_\mu - p_\mu^{\mathrm{F}}$ has degree smaller than $|\mu|-1$; see \cite[Theorems 7.4 and 7.13]{Mel17}.

\begin{proposition}\label{prop:jucys_murphy}
Given an integer partition $\lambda$ and $k \geq 1$, we denote
$$p_k^\oblong(\lambda) = \sum_{\oblong \in \lambda} (c(\oblong))^k$$
the $k$-th power sum of the contents of the cells of $\lambda$. This function is an observable of Young diagrams with degree $k+1$, and for any $k \geq 1$,
$$p_k^\oblong = \frac{\varSigma_{k+1}}{k+1} + \frac{1}{2}\sum_{j=1}^{k-1} \varSigma_{j}\varSigma_{k-j} + \text{terms with degree smaller than }k-1.$$
\end{proposition}

\begin{proof}
Note that power sums of the contents have been studied extensively in \cite{LT01}; hereafter, we prove the formula for the leading terms of $p_k^\oblong$ by reinterpreting $\obs$ as the Ivanov--Kerov algebra of partial permutations \cite{IK99}. Given two finite subsets $A$ and $B$ of $\N^*$ and two permutations $\sigma \in \sym(A)$ and $\tau \in \sym(B)$, we can define the product of the two pairs $(\sigma,A)$ and $(\tau,B)$ by
$$(\sigma,A) \,(\tau,B) = (\sigma\circ \tau,A\cup B).$$
This definition gives rise to a semigroup $\mathscr{P}$ whose elements are the pairs $(\sigma,A)$ with $A$ finite subset of $\N^*$ and $\sigma \in \sym(A)$. These pairs are called \emph{partial permutations}, and a filtration of $\mathscr{P}$ is provided by $\deg (\sigma,A) = |A|$. We continue to denote $\mathscr{P}$ the algebra whose elements are the formal linear combinations of partial permutations with bounded degree. For any integer partition $\mu$ with size $k$, set 
$$\varSigma_\mu = \sum_{i_1 \neq i_2 \neq \cdots \neq i_k} \big((i_1,\ldots,i_{\mu_1})(i_{\mu_1+1},\ldots,i_{\mu_1+\mu_2})\cdots,\{i_1,\ldots,i_k\}\big),$$
the sum running over $k$-arrangements of positive integers. The linear span of the elements $\varSigma_\mu$ is the subalgebra $\mathscr{P}^{\sym(\infty)}$ of elements of $\mathscr{P}$ which are invariant by the conjugation action of $\sym(\infty)$:
$$\tau \cdot (\sigma,A) = (\tau\sigma\tau^{-1},\tau(A)).$$
Then, the identification of the symbols $\varSigma_\mu$ in $\mathscr{P}^{\sym(\infty)}$ and in $\obs$ gives rise to an isomorphism of filtered algebras; see \cite[Section 7.3]{Mel17}. 
Let us now introduce the \emph{generic Jucys--Murphy elements}
$$X_1 = 0 \quad;\quad X_{n \geq 2} = \sum_{i=1}^{n-1} \big((i,n),\{i,n\}\big);$$
they are elements of degree $2$ in the algebra $\mathscr{P}$ of partial permutations. For any $k \geq 1$,
$$e_k(X_1,X_2,X_3,\ldots,X_n,\ldots) = \sum_{\mu \in \ym(k)} \frac{\varSigma_{\mu + 1}}{z_{\mu+ 1}},$$
where $\mu + 1$ denotes the integer partition $(\mu_1+1,\mu_2+1,\ldots,\mu_\ell + 1)$;
see \cite[Lemma 8.24]{Mel17}. For any $n \geq 1$, the linear map
\begin{align*}
\pi_n : \mathscr{P} &\to \C\sym(n) \\ 
(\sigma,A) &\mapsto \begin{cases}
    \sigma &\text{if }A \subset \lle 1,n\rre \\
    0&\text{otherwise}
\end{cases}
\end{align*}
is a morphism of algebras, and it sends $\mathscr{P}^{\sym(\infty)}$ to the center $Z(\C\sym(n))$. We then have, for any integer partition $\lambda \in \ym(n)$,
$$\varSigma_\mu(\lambda) = \chi^\lambda(\pi_n(\varSigma_\mu)),$$
where on the left-hand side we evaluate the observable $\varSigma_\mu$ on the integer partition $\lambda$, and on the right-hand side we consider the symbol $\varSigma_\mu \in \mathscr{P}^{\sym(\infty)}$, project it onto $Z(\C\sym(n))$ by the map $\pi_n$, and then evaluate the normalised irreducible character $\chi^\lambda$ on this linear combination of conjugacy classes. Now, if $m \leq n$, then $\pi_n(X_m) = J_m$ is a special element of $\C\sym(n)$: this Jucys--Murphy  element acts diagonally on a certain basis $(b_T)_{T \in \ST(\lambda)}$ of the irreducible representation of $\sym(n)$ with label $\lambda$, and more precisely,
$$J_m\,b_T = c(m,T)\,b_T,$$
$c(m,T)$ being the content of the cell labeled $m$ in the standard tableau $T$; see \cite{OV04} and \cite[Theorem 8.14]{Mel17}. Therefore, for any $k \geq 1$, the elementary symmetric function of the contents 
$$e_k^\oblong(\lambda) = \sum_{\oblong_1 < \oblong_2 <\cdots < \oblong_k} c(\oblong_1)\,c(\oblong_2)\cdots c(\oblong_k)$$
is an observable in the algebra $\obs$ (in the formula above, we have fixed an arbitrary total order on the cells of the Young diagram $\lambda$, for instance the one coming from a labeling of the cells by a standard tableau). Indeed, if we endow the irreducible representation with label $\lambda$ with the scalar product for which $(b_T)_{T \in \ST(\lambda)}$ is an orthonormal basis, then
\begin{align*}
e_k^\oblong(\lambda) &= \frac{1}{\dim \lambda}\sum_{T \in \ST(\lambda)} \sum_{m_1<m_2<\cdots<m_k} \scal{c(m_1,T)\,c(m_2,T)\cdots c(m_k,T)\, b_T}{b_T} \\ 
&= \frac{1}{\dim \lambda} \sum_{T \in \ST(\lambda)} \sum_{m_1<m_2<\cdots<m_k} \scal{J_{m_1}J_{m_2}\cdots J_{m_k}\, b_T}{b_T} \\ 
&= \chi^\lambda(e_k(J_1,J_2,\ldots,J_n)) = \chi^\lambda(\pi_n(e_k(X_1,X_2,\ldots))) = \sum_{\mu \in \ym(k)} \frac{\varSigma_{\mu+1}(\lambda)}{z_{\mu+1}},
\end{align*}
by using the aforementioned expansion of the elementary function of the generic Jucys--Murphy elements. The relation 
$$p_k = (-1)^{k}\, k \sum_{\mu \in \ym(k)} \frac{(-1)^{\ell(\mu)-1}\,(\ell(\mu)-1)!}{\prod_{s \geq 1}(m_s(\mu))!}\, e_\mu$$
between the power sums $p_k$ and the elementary symmetric functions $e_k$ in the algebra $\Sym$ translates then to the same relation between the functions $p_k^\oblong$ and $e_k^{\oblong}$, and therefore proves that the functions $p_k^\oblong$ belong to $\obs$. Moreover, for the same reason as above,
$$p_k^\oblong(\lambda) = \chi^\lambda(\pi_n(p_k(X_1,X_2,\ldots))),$$
so in order to prove the proposition, it now suffices to compute the $k$-th power sum of the generic Jucys--Murphy elements. We have
$$(X_n)^k=\sum_{i_1,\ldots,i_k<n} \big((i_1,n)(i_2,n)\cdots(i_k,n),\{i_1,i_2,\ldots,i_k,n\}\big),$$
so $\deg (X_n)^k \leq k+1$ for any $k \geq 1$. More precisely, the elements with degree $k+1$ are those such that $i_1\neq i_2 \neq \cdots \neq i_k$, and we then have
$(i_1,n)(i_2,n)\cdots(i_k,n) = (n,i_k,\ldots,i_2,i_1)$. Therefore, the component with degree $k+1$ of $(X_n)^k$ is
$$\sum_{\substack{i_1,\ldots,i_k < n \\ i_1\neq i_2 \neq \cdots \neq i_k}} \big((i_k,\ldots,i_1,n),\{i_1,\ldots,i_k,n\}\big).$$
Similarly, the component with degree $k$ of $(X_n)^k$ is the sum of the product of transpositions $(i_1,n),\ldots,(i_k,n)$ with exactly one pair $(k_1,k_2)$ such that $1 \leq k_1<k_2 \leq k$ and $i_{k_1}=i_{k_2}$; all the other indices are distinct. Since
\begin{align*}
(i_1,n)(i_2,n)\cdots (i_{k_1},n) \cdots (i_{k_2},n) \cdots (i_k,n) &= (i_{k_2-1},\ldots,i_{k_1},\ldots,i_1,n) (i_{k},\ldots,i_{k_2+1},i_{k_1},n) \\ 
&= (i_k,\ldots,i_{k_2+1},i_{k_1-1},\ldots,i_1,n)(i_{k_2-1},i_{k_2-2},\ldots,i_{k_1})
\end{align*}
is a disjoint product of a $k-(k_2-k_1)$ cycle with a $k_2-k_1$ cycle, we see that the terms with degree $k$ in $(X_n)^k$ are given by:
\begin{align*}
&=\sum_{1 \leq k_1<k_2 \leq k} \sum_{\substack{a_1\neq a_2 \neq \cdots \neq a_{k-(k_2-k_1)-1}\neq b_1 \neq \cdots \neq b_{k_2-k_1} \\ a_x,b_y <n}} \big((a_1,\ldots,a_{k-{k_2-k_1}-1},n) (b_1,\ldots,b_{k_2-k_1}),\{a_x,b_y,n\}\big) \\ 
&= \sum_{j=1}^{k-1} (k-j) \sum_{\substack{a_1\neq a_2 \neq \cdots \neq a_{k-j-1}\neq b_1 \neq \cdots \neq b_{j} \\ a_x,b_y <n}} \big((a_1,\ldots,a_{k-j-1},n)(b_1,\ldots,b_j),\{a_x,b_y,n\}\big).
\end{align*}
Let us now take the sum over $n \geq 1$ of the terms with degree $k+1$ and $k$. The homogeneous component with degree $k+1$ of $p_k(X_1,X_2,\ldots,X_n,\ldots)$ is
$$\sum_{\substack{i_1\neq i_2 \neq \cdots \neq i_k \neq i_{k+1} \\ i_{k+1} = \max(i_1,\ldots,i_{k+1})}} \big((i_{k+1},\ldots,i_2,i_1),\{i_1,i_2,\ldots,i_{k+1}\} \big) = \frac{\varSigma_{k+1}}{k+1},$$
the factor $\frac{1}{k+1}$ taking into account the choice of a maximum in a $(k+1)$-cycle. Similarly, 
the homogeneous component with degree $k$ in $p_k(X_1,X_2,\ldots,X_n,\ldots)$ is 
\begin{align*}
&= \sum_{j=1}^{k-1} (k-j) \sum_{\substack{a_1\neq a_2 \neq \cdots \neq a_{k-j-1} \neq a_{k-j}\neq b_1 \neq \cdots \neq b_{j} \\ a_{k-j} \text{ is the maximum}}} \big((a_1,\ldots,a_{k-j-1},a_{k-j})(b_1,\ldots,b_j),\{a_x,b_y\}\big) \\ 
&= \frac{1}{2} \sum_{j=1}^{k-1} \varSigma_{(j,k-j)},
\end{align*}
with by convention $\varSigma_{(j,k-j)} = \varSigma_{(k-j,j)}$ if $j<\frac{k}{2}$.
Indeed, the factors $(k-j)$ are compensated by the choice of a maximum in one cycle, and the factor $\frac{1}{2}$ corresponds to the choice of one of the two cycles which contains the maximum. Notice that $\varSigma_{(j,k-j)} - \varSigma_j\varSigma_{k-j}$ has degree smaller than $k-1$; this follows from a more general rule for the product of symbols $\varSigma_\mu$, see \cite[Proposition 7.6]{Mel17}.
The claimed formula for the leading terms of $p_k^\oblong$ is now obtained by using the projection $\pi_n$ of $\mathscr{P}^{\sym(\infty)}$ onto $Z(\C\sym(n))$ and the evaluation of a normalised irreducible character $\chi^\lambda$.
\end{proof}

\begin{corollary}
For any $k \geq 1$, 
$$p_k^\oblong = \frac{p_{k+1}^{\mathrm{F}}}{k+1} + \text{terms with degree smaller than }k-1.$$
\end{corollary}

\begin{proof}
This follows immediately from Proposition \ref{prop:jucys_murphy} and from the following expansion of the symbols $\varSigma_k$ in the basis of Frobenius moments: for any $k \geq 2$,
$$\varSigma_k = p_k^{\mathrm{F}} -\frac{k}{2}\sum_{j=1}^{k-2}p_{j}^{\mathrm{F}}\,p_{k-1-j}^{\mathrm{F}} + \text{terms with degree smaller than }k-2.$$
Let us explain the origin of this formula.The Frobenius--Schur formula for the characters of the symmetric group leads to the following explicit formula for $\varSigma_k(\lambda)$, which is due to Wassermann (see \cite{Was81} and \cite[Proof of Theorem 7.13]{Mel17}):
$$\varSigma_k(\lambda) = [z^{-1}]\left(-\frac{1}{k}\,\left(z-\frac{1}{2}\right)^{\downarrow k}\,\prod_{i=1}^n \frac{z-(-i+\frac{1}{2})}{z-(\lambda_i-i+\frac{1}{2})}\,\prod_{i=1}^n \frac{z-k-(\lambda_i-i+\frac{1}{2})}{z-k-(-i+\frac{1}{2})}\right).$$
The right-hand side rewrites as:
\begin{align*}
 &-\frac{1}{k}\,\left(z-\frac{1}{2}\right)^{\downarrow k} \exp\left(\sum_{j=1}^\infty \frac{p_j^\mathrm{F}}{j\,z^j} \left(1 - \frac{1}{(1-\frac{k}{z})^j}\right)\right) \\ 
  &= -\frac{1}{k}\,\left(z-\frac{1}{2}\right)^{\downarrow k} \exp\left(-\sum_{j=1}^\infty\sum_{m=1}^\infty \binom{j+m-1}{m}\frac{p_j^\mathrm{F}\,k^m}{j\,z^{j+m}} \right).\end{align*}
Set $c_{j,m} = -\binom{j+m-1}{m} \frac{k^m}{j}$. We have
$$-\frac{1}{k}\,\left(z-\frac{1}{2}\right)^{\downarrow k} = \sum_{l=0}^{k-1} \frac{(-1)^{l-1} e_l(\frac{1}{2},\frac{3}{2},\ldots,\frac{2k-1}{2})}{k}\,z^{k-l},$$
so 
\begin{align*}
\varSigma_k = \sum_{r=1}^\infty \sum_{l=0}^{k-1} \frac{(-1)^{l-1} e_l(\frac{1}{2},\frac{3}{2},\ldots,\frac{2k-1}{2})}{k\,r!} \sum_{\substack{k+1-l = j_1+m_1+\cdots + j_r+m_r \\ j_1,m_1,\ldots,j_r,m_r \geq 1}} c_{j_1,m_1}\cdots c_{j_r,m_r} \left(\prod_{s=1}^r p_{j_s}^\mathrm{F}\right).
\end{align*}
Let us now extract the terms with the highest degree:
\begin{itemize}
    \item the maximum degree is obtained when $r=1$, $l=0$, $j_1=k$ and $m_1=1$. We then get
    $$\frac{(-1)^{-1}}{k} \,c_{k,1}\,p_k^\mathrm{F} = p_k^\mathrm{F}.$$
    \item the terms with degree $k-1$ are obtained when $r=1$, $l=0$, $j_1=k-1$ and $m_1=2$; when $r=1$, $l=1$, $j_1=k-1$ and $m_1=1$; and also when $r=2$, $l=0$, $j_1+j_2=k-1$ and $m_1=m_2=1$. We get respectively
    \begin{align*}
    &\frac{(-1)^{-1}}{k}\,c_{k-1,2}\,p_{k-1}^\mathrm{F} = \frac{k^2}{2} \,p_{k-1}^\mathrm{F};\\
    &\frac{(-1)^0\,e_1(\frac{1}{2},\frac{3}{2},\ldots,\frac{2k-1}{2})}{k}\,c_{k-1,1}\, p_{k-1}^\mathrm{F} = - \frac{k^2}{2}\, p_{k-1}^\mathrm{F};
    \end{align*}
    and the sum over indices $j \in \lle 1,k-2\rre$ of 
    $$
    \frac{(-1)^{-1}}{2k}\,c_{j,1}\,c_{k-1-j,1}\,p_{j}^\mathrm{F}\,p_{k-1-j}^\mathrm{F} = - \frac{k}{2}\,p_{j}^\mathrm{F}\,p_{k-1-j}^\mathrm{F}.
    $$
\end{itemize}
Taking the sum of these terms ends the proof of our corollary.
\end{proof}
\medskip

The corollary above implies that 
\begin{equation}
\beta_r = \sum_{s=0}^r \binom{r}{s}\,\frac{(p_{1}^{\mathrm{F}})^s\,p_{r+1-s}^{\mathrm{F}}}{r+1-s}  +\text{terms with degree smaller than }r-1.\label{eq:beta_r}\tag{$\beta$}
\end{equation}
The following lemma is the corresponding control on $\beta(z)$:
\begin{lemma}\label{lem:control_betaz}
If $z \in \frac{1}{2}\mathscr{D}_0$, then
$$\left|\beta(z) - n \int_{t=0}^1 \int_{x=-1}^1  \left(\varphi(txz)+\frac{\varphi((x+t)z)-\varphi(tz)}{x}\right)\mu\expn(\!\DD{x})\DD{t} \right|,$$
is locally uniformly bounded, independently from the growing sequence $(\lambda\expn)_{n\geq 1}$.
\end{lemma}
\begin{proof}
Recall that $\beta(z) = \sum_{\oblong \in \lambda} \varphi((1+\frac{c(\oblong)}{n})z)$. We split this sum in $2d$ parts, according to which hook with size $a_i+b_i$ and based at the $i$-th cell of the diagonal of $\lambda$ contains $\oblong$:
$$\beta(z) = \sum_{i=1}^d \left(\frac{1}{2}f(0) + \sum_{k=1}^{a_i-\frac{1}{2}} f\!\left(\frac{k}{n}\right)\right)+ \sum_{i=1}^d \left(\frac{1}{2}f(0) + \sum_{k=1}^{b_i-\frac{1}{2}} f\!\left(-\frac{k}{n}\right)\right),$$
with $f(t) = \varphi((1+t)z)$. For any $i \in \lle 1,d\rre$, we have
\begin{align*}
\int_0^1 a_i\,f\!\left(\frac{a_i t}{n}\right)\DD{t} = \int_0^{a_i} f\!\left(\frac{u}{n}\right)\DD{u} &=  \int_{0}^{\frac{1}{2}} f\!\left(\frac{u}{n}\right)\DD{u} + \sum_{k=1}^{a_i-\frac{1}{2}} \int_{k-\frac{1}{2}}^{k+\frac{1}{2}} f\!\left(\frac{u}{n}\right)\DD{u} \\ 
&= \frac{1}{2}f(0) + \sum_{k=1}^{a_i-\frac{1}{2}} f\!\left(\frac{k}{n}\right) + O\!\left(\frac{\|f'\|_\infty}{n} + \frac{a_i\|f''\|_\infty}{n^2}\right),
\end{align*}
and similarly for the parameters $b_i$. As a consequence, 
$$\beta(z) = n\int_{t=0}^1\int_{x=-1}^1 \varphi((1+xt)z) \,\mu\expn(\!\DD{x})\DD{t} + O(1)$$
with a locally uniformly bounded remainder. Let us then show that the integral is the same as the integral of the statement of the lemma. Since both integrals are holomorphic functions on $\frac{1}{2}\mathscr{D}_0$, it suffices to prove the equality on a small disc containing $0$, and then we can expand the power series $\varphi$. We have:
\begin{align*}
I_1 &= \int_{t=0}^1\int_{x=-1}^1 \varphi((1+xt)z) \,\mu\expn(\!\DD{x})\DD{t} \\ 
&= \sum_{r=1}^\infty \frac{B_r}{r}\,\frac{z^r}{r!} \int_{t=0}^1 \int_{x=-1}^1 (1+xt)^r \,\mu\expn(\!\DD{x})\DD{t} \end{align*}
and 
\begin{align*}
I_2&=\int_{t=0}^1 \int_{x=-1}^1  \left(\varphi(txz)+\frac{\varphi((x+t)z)-\varphi(tz)}{x}\right)\mu\expn(\!\DD{x})\DD{t} \\ 
&= \sum_{r=1}^\infty \frac{B_r}{r}\,\frac{z^r}{r!} \int_{t=0}^1 \int_{x=-1}^1 \left((tx)^r + \frac{(x+t)^r -t^r}{x}\right) \mu\expn(\!\DD{x})\DD{t},
\end{align*}
so we have to show that for any $r \geq 1$, 
$$\int_{t=0}^1 \int_{x=-1}^1 (1+xt)^r \,\mu\expn(\!\DD{x})\DD{t} =\int_{t=0}^1 \int_{x=-1}^1 \left((tx)^r + \frac{(x+t)^r -t^r}{x}\right) \mu\expn(\!\DD{x})\DD{t}.$$
The first integral is equal to $\sum_{s=0}^r \binom{r}{s} \frac{p_{r+1-s}(\omega\expn)}{r+1-s}$, and the second integral is 
$$\frac{p_{r+1}(\omega\expn)}{r+1} + \sum_{s=0}^{r-1} \binom{r}{s} \frac{p_{r-s}(\omega\expn)}{s+1},$$
so the identity simply comes from $\binom{r}{s}\frac{1}{s+1} = \binom{r}{s+1}\frac{1}{r-s}$.
\end{proof}
\medskip

\subsection{Asymptotics of the cumulants}\label{sub:asymptotics_of_the_cumulants}
We can now proceed to:
\begin{proof}[Proof of Theorems \ref{thm:asymptotics_cumulants} and \ref{thm:cumulants_are_observables}]
The previous arguments have shown that for $r$ even, $\alpha_r$, $\beta_r$ and $\gamma_r$ are observables of the Young diagram $\lambda$; therefore, the same is true for $\kappa^{(r)}(\majT)= \frac{B_r}{r}(\alpha_r - \beta_r - \gamma_r)$. Moreover, by taking the adequate linear combination of the three equations \eqref{eq:alpha_r}, \eqref{eq:beta_r} and \eqref{eq:gamma_r}, we get:
\begin{align*}
\frac{r}{B_r}\,\kappa^{(r)}(\majT) &= \frac{(p_1^{\mathrm{F}})^{r+1} - p_{r+1}^{\mathrm{F}}}{r+1} +\frac{1}{2}\left((p_1^{\mathrm{F}})^r + \sum_{s=1}^{r-1} \binom{r}{s} (-1)^s\,p_s^{\mathrm{F}}\,p_{r-s}^{\mathrm{F}}\right)+\cdots
\end{align*}
where the dots indicate terms with degree smaller than $r-1$. So, the terms with degree $r+2$ disappear, and the terms with degree $r+1$ and $r$ simplify as indicated above.
\end{proof}

\begin{example}\label{ex:second_cumulant}
By using a computer algebra system, we can compute: 
\begin{align*}
\kappa^{(2)}(\majT) &= \frac{1}{36}\,(\varSigma_{1^{(3)}}- \varSigma_{3}) = \frac{1}{36}\,\left((p_1^{\mathrm{F}})^3 - p_3^{\mathrm{F}} - \frac{3}{2}\, (p_1^{\mathrm{F}})^2 + \frac{3}{4}\,p_1^{\mathrm{F}}\right) .
\end{align*}
\end{example} 

\begin{remark}\label{rem:first_cumulant}
The expectation of $\majT$ (first cumulant) can also be written as an observable of the integer partition $\lambda$. Indeed,
\begin{align*}
\esper[\majT] &= b(\lambda) +\frac{1}{2}\left(\sum_{i=1}^n i - \sum_{\oblong \in \lambda} h(\oblong)\right) \\
&= \sum_{i=1}^n (i-1)\lambda_i + \frac{1}{2} \left(\sum_{1\leq i<j\leq n}(\lambda_i-\lambda_j + j-i) - \sum_{\oblong \in \lambda} (n+c(\oblong)) - \sum_{i=1}^n (n-i-1)\,i\right) \\ 
&= \sum_{i=1}^n (i-1)\lambda_i + \frac{1}{2} \left(\sum_{1\leq i<j\leq n}(\lambda_i-\lambda_j)  -\frac{n(n-1)}{2} - \frac{1}{2}\,p_2^{\mathrm{F}}(\lambda) \right) \\
&= \frac{n-1}{2}\left(\sum_{i=1 }^n \lambda_i \right) -\frac{n(n-1)}{4} - \frac{1}{4}\,p_2^{\mathrm{F}}(\lambda) = \frac{n(n-1)}{4}-\frac{1}{4}\,p_2^{\mathrm{F}}(\lambda),
\end{align*}
so $\kappa^{(1)}(\majT) = \frac{1}{4} ((p_1^{\mathrm{F}})^2 - p_2^{\mathrm{F}} - p_1^{\mathrm{F}}) = \frac{1}{4}\,(\varSigma_{1^{(2)}} - \varSigma_2)$.
\end{remark}

\begin{remark}
The calculation of the previous remark allows one to compute the range of $\majT$ when $T$ runs over $\ST(\lambda)$. Indeed, by Lemma \ref{lem:fundamental},
$$0 \leq \majT - b(\lambda) \leq \sum_{i=1}^n i - \sum_{\oblong \in \lambda}h(\oblong),$$
which can be rewritten as:
$$b(\lambda) \leq \majT \leq  2\,\esper[\majT] - b(\lambda).$$
The left-hand is $\sum_{i=1}^{\ell(\lambda)} (i-1)\lambda_i $, and the right-hand side can be rewritten as $\frac{n(n-1)}{2} - \sum_{i=1}^{\ell(\lambda)} \frac{\lambda_i(\lambda_i-1)}{2}$, for instance by using the identity $\varSigma_2 = 2\,p_1^\oblong$.
\end{remark}
\medskip

The upper bounds computed in Lemmas \ref{lem:control_gammaz}, \ref{lem:control_alphaz} and \ref{lem:control_betaz} can also be combined in order to get:
\begin{theorem}[Control of the log-Laplace transform]\label{thm:control_laplace}
On the domain $\frac{1}{2}\mathscr{D}_0$, the analytic function 
$$\log \esper\!\left[\E^{\frac{z\,\majTn}{n}}\right] - n\int_{t=0}^1 \int_{x=-1}^1 (\varphi(tz)-\varphi(txz)) \,\mu\expn(\!\DD{x})\DD{t}$$
is locally uniformly bounded, independently from the growing sequence $(\lambda\expn)_{n \geq 1}$.
\end{theorem}
\bigskip

\section{Asymptotics of the distribution of \texorpdfstring{$\majT$}{majT}}
In this section, we prove our main Theorems \ref{main:log_laplace}, \ref{main:large_deviations} and \ref{main:berry_esseen} by using the controls on the cumulants and on the log-Laplace transform provided by Theorems \ref{thm:asymptotics_cumulants} and \ref{thm:control_laplace}. Our techniques are inspired by the framework of mod-$\phi$ convergent sequences developed in \cite{JKN11,KN12,DKN15,FMN16,FMN19,BMN19,MN22}. However, the knowledge of this general framework is not required, and hereafter, we shall give ad hoc proofs of the asymptotic estimates satisfied by the distribution of $\majTn$ with $T\expn \sim \mathcal{U}(\ST(\lambda\expn))$.
\medskip

\subsection{Asymptotics of the log-Laplace transform}\label{sub:proof_A}
The proof of Theorem \ref{main:log_laplace} relies on a resummation of the estimates of the cumulants, and on the following argument from complex analysis.

\begin{lemma}\label{lem:clever_complex}
Let $(f_n)_{n \geq 1}$ be a sequence of holomorphic functions on a domain $\mathscr{D}$ (connected open subset of the complex plane $\C$), such that $(f_n)_{n \geq 1}$ converges to $0$ on an open subset $\mathscr{D}'\subset \mathscr{D}$. We also suppose that $(f_n)_{n \geq 1}$ is uniformly bounded on any compact subset of $\mathscr{D}$. Then, $(f_n)_{n \geq 1}$ converges to $0$ locally uniformly on $\mathscr{D}$.
\end{lemma}

\begin{proof}
By the Montel theorem, $(f_n)_{n \geq 1}$ is sequentially compact on $\mathscr{D}$, and any limit of a convergent subsequence $(f_{n_k})_{k \geq 1}$ vanishes on $\mathscr{D}'$, hence on the whole domain $\mathscr{D}$ by the principle of analytic continuation. So, the limit of convergent subsequences is unique, which proves that $(f_n)_{n \geq 1}$ converges to the zero function on $\mathscr{D}$.
\end{proof}

\begin{lemma}\label{lem:continuity_lambda_psi}
Given a domain $\mathscr{D}\subset \C$, we denote $\mathscr{O}(\mathscr{D})$ the space of holomorphic functions on this domain, endowed with the Montel topology of locally uniform convergence. The maps
\begin{align*}
 \Lambda : \Omega &\to \mathscr{O}\!\left(\frac{1}{2}\mathscr{D}_0\right)\qquad \text{and}\qquad \Psi : \Omega \to \mathscr{O}\!\left(\frac{1}{2}\mathscr{D}_0\right) \\ 
 \omega &\mapsto \Lambda_\omega\hspace{4.3cm}\omega \mapsto \Psi_\omega
 \end{align*} 
are continuous with respect to this topology and to the topology of convergence of coordinates in the Thoma simplex $\Omega$.
\end{lemma}

\begin{proof}
Let us treat for instance the case of $\Lambda$. We compose the following continuous maps:
\begin{enumerate}
    \item the map $\omega \in \Omega \mapsto \mu_\omega \in \mathscr{M}^{1}([-1,1])$ introduced during the proof of Lemma \ref{lem:frobenius_thoma};
    \item the map $\mu \in \mathscr{M}^{1}([-1,1]) \mapsto \mu \otimes \DD{t} \in \mathscr{M}^1([-1,1]\times [0,1])$;
    \item the integration map 
    \begin{align*}
     F : \mathscr{M}^1([-1,1] \times [0,1]) &\to \mathscr{O}\!\left(\frac{1}{2}\mathscr{D}_0\right) \\ 
     \nu &\mapsto \int_{[-1,1]\times [0,1]} (\phi(tz)-\phi(txz))\,\nu(\!\DD{x}\DD{t}).
     \end{align*} 
\end{enumerate}
Let us detail the last item. For $z$ fixed in $\frac{1}{2}\mathscr{D}_0$, $\nu \mapsto F(\nu)(z)$ is continuous from $\mathscr{M}^1([-1,1] \times [0,1])$ to $\C$, by definition of the weak topology of convergence of probability measures. By the same argument as in the proof of Lemma \ref{lem:control_alphaz}, if $z$ stays in a compact subset $K \subset \frac{1}{2}\mathscr{D}_0$, then there is a uniform bound on $$\{(\phi(tz)-\phi(txz))\,\,|\,\,z \in K,\,\,t \in [0,1],\,\,x \in [-1,1]\},$$  and therefore on $\{F(\nu)(z)\,\,|\,\,z\in K,\,\, \nu \in \mathscr{M}^1([-1,1]\times [0,1])\}$. The Montel compactness principle ensures then the global continuity of $\nu \mapsto F(\nu)$.
\end{proof}

\begin{remark}
We are not saying that $\frac{1}{2}\mathscr{D}_0$ is the whole domain of definition and analyticity of the integrals $\Lambda_\omega$ and $\Psi_\omega$. Indeed, $\phi$ admits a logarithmic singularity at $2\I \pi$, but this singularity can be removed by integration of the parameter $t$; see also our Remark \ref{rem:not_log_laplace}, where we give a function $\Lambda_{\omega_0}$ which is well-defined on the whole imaginary line $\I\R$. 
\end{remark}

\begin{proof}[Proof of Theorem \ref{main:log_laplace}] 
We apply Lemma \ref{lem:clever_complex} to the sequence of functions
$$f_n(z) = \log \esper\!\left[\E^{z(X\expn-\esper[X\expn])}\right] - n\,\Lambda_{\omega\expn}(z) - \Psi_{\omega\expn}(z).$$
Let us remark the following consequence of Lemma \ref{lem:continuity_lambda_psi}: since $\Omega$ is a compact topological set, the functions $\Psi_{\omega\expn}(z)$ are locally uniformly bounded on $\frac{1}{2}\mathscr{D}_0$, independently from the growing sequence $(\lambda\expn)_{n \geq 1}$. Therefore, in order to prove the theorem, it suffices to show that 
\begin{enumerate}
    \item $f_n(z)$ converges to $0$ on a disc containing the origin;
    \item $\log \esper\!\left[\E^{z(X\expn-\esper[X\expn])}\right] - n\,\Lambda_{\omega\expn}(z)$ is locally uniformly bounded on $\frac{1}{2}\mathscr{D}_0$.
\end{enumerate}
The second item is an immediate consequence of Theorem \ref{thm:control_laplace}, by using the computation from Remark \ref{rem:first_cumulant} in order to deal with the term $$-z\,\esper[X\expn] = -\frac{nz}{4}\left(1-\int_{x=-1}^1 x\,\mu\expn(\!\DD{x})\right)+O(1)$$
in the log-Laplace transform of the recentered random variables (this explains the replacement of $\varphi$ by $\phi$ in the formula for $\Lambda_\omega$).
We now suppose that $z$ belongs to the disc $D(0,\pi)$. If $\lambda\expn$ is an integer partition with size $n$, then we have 
$$p_k^{\mathrm{F}}(\lambda\expn) = n^k \,p_k(\omega\expn) = n^k \int_{x=-1}^1\,x^{k-1}\,\mu\expn(\!\DD{x}),$$
so the estimate of Theorem \ref{thm:asymptotics_cumulants} rewrites as:
\begin{align*}
\kappa^{(r)}(X\expn) &= \frac{n\,B_r}{r(r+1)} \left(\int_{x=-1}^1 (1-x^{r})\,\mu\expn(\!\DD{x}) \right)\\ 
&\,\,\, + \frac{B_r}{2r} \left(1 + \sum_{s=1}^{r-1} \binom{r}{s} (-1)^s \int_{x=-1}^1\int_{y=-1}^1 x^{s-1}\,y^{r-s-1} \,\mu^{(n),\otimes 2}(\!\DD{x}\DD{y})\right) + O(n^{-1})
\end{align*}
for any $r \geq 2$. Let us denote $n\,K_r$ the term proportional to $n$ in each expansion above. We have:
\begin{align*}
\sum_{r=2}^\infty K_r\,\frac{z^r}{r!} &= \sum_{r=2}^\infty \frac{B_r}{r(r+1)} \left(\int_{x=-1}^1 \frac{z^r(1-x^{r})}{r!}\,\mu\expn(\!\DD{x})\right) \\ 
&=\int_{t=0}^1 \int_{x=-1}^1 \left(\sum_{r=2}^\infty \frac{B_r}{r}\,\frac{(tz)^r(1-x^r)}{r!}\right)
\,\mu\expn(\!\DD{x})\DD{t} \\ 
&= \int_{t=0}^1 \int_{x=-1}^1 (\phi(tz)-\phi(txz)) \,\mu\expn(\!\DD{x}) \DD{t}.
\end{align*}
The exchange of the symbols $\int$ and $\sum$ is justified by the fact that we are looking at convergent power sums. Similarly, if $L_r$ is the term of order $O(1)$ in the expansion of the cumulant $\kappa^{(r)}(X\expn)$, then
\begin{align*}
2\sum_{r=2}^\infty L_r\,\frac{z^r}{r!} &= \phi(z) + \sum_{r=2}^\infty \frac{B_r}{r} \,\frac{z^r}{r!} \int_{x=-1}^1 \int_{y=-1}^1 \frac{(y-x)^r-y^r-(-x)^r}{xy}\,\mu^{(n),\otimes 2}(\!\DD{x}\DD{y}) \\ 
&= \int_{x=-1}^1 \int_{y=-1}^1 \left(\phi(z) + \frac{\phi((y-x)z)-\phi(yz)-\phi(-xz)}{xy}\right) \,\mu^{(n),\otimes 2}(\!\DD{x}\DD{y}).
\end{align*}
So, one obtains the formula of Theorem \ref{main:log_laplace}. The only thing that we need to check is that we can indeed exchange the limits with the summation over indices $r \geq 1$. However, by Theorem \ref{thm:control_laplace},
$$\log \esper\!\left[\E^{z (X\expn-\esper[X\expn])}\right] - n \int_{t=0}^1 \int_{x=-1}^1 (\phi(tz)-\phi(txz)) \,\mu\expn(\!\DD{x}) \DD{t}$$
 is locally uniformly bounded for $|z|<\pi$. Therefore, the remainder lives in a compact subset of the space of holomorphic functions on $D_{(0,\pi)}$, and we can indeed take the sum of the limits.
\end{proof} 

\begin{remark}\label{rem:not_log_laplace}
Our Theorem \ref{main:log_laplace} ensures that given a convergent sequence $(\lambda\expn)_{n \geq 1}$ with limit $\omega$, the log-Laplace transform of $X\expn$ behaves asymptotically like $n\,\Lambda_\omega(z)$. This makes one wonder whether $\Lambda_\omega(z)$ is itself the log-Laplace transform of a probability distribution. The general answer to this question is negative. Indeed, by Bochner's theorem, if $\Lambda(z)$ is the log-Laplace transform of a probability distribution, then for any finite family of real numbers $\xi_1,\ldots,\xi_r$, the matrix
$$(\E^{\Lambda(\I\xi_i-\I\xi_j)})_{1\leq i,j\leq r}$$
is Hermitian non-negative definite. Consider the case where $\omega_0=((0,0,\ldots),(0,0,\ldots))$ is the Thoma parameter with two null sequences (by continuity, any Thoma parameter sufficiently close to this parameter will also fail the Bochner condition). We then have:
$$\Lambda_{\omega_0}(\I \xi) = \int_{t=0}^1 \log\left(\frac{\sin \frac{t\xi}{2}}{\frac{t\xi}{2}}\right)\DD{t}$$
since $\mu_{\omega_0} = \delta_0$. With $(\xi_1,\xi_2,\xi_3)=(0,3,6)$, a numerical integration of the function above yields a symmetric matrix $(\E^{\Lambda(\I\xi_i-\I\xi_j)})_{1\leq i,j\leq 3}$ whose smallest eigenvalue is $-0.0135\ldots$, so it is not non-negative definite.
\end{remark}
\medskip

\subsection{Upper bounds on cumulants and Berry--Esseen estimates}\label{sub:proof_B}
In \cite{FMN16,FMN19}, the classical method of cumulants used in order to prove a central limit theorem has been perfected in order to obtain moderate deviation estimates and upper bounds on the Kolmogorov distance between a distribution and its normal approximation. In our setting, these methods translate to the following:

\begin{lemma}\label{lem:berry}
Let $X$ be a random variable with variance $\sigma^2>0$. We suppose that for any $r \geq 3$,
$$|\kappa^{(r)}(X)| \leq r! \,\sigma^2\,K L^{r-2}$$
for some constants $K\geq \frac{1}{4}$ and $L > 0$; in particular, the log-Laplace transform $\log \esper[\E^{xZ}]$ is convergent on the open disc $D_{(0,\frac{1}{L})}$. Then, 
$$\dkol\left(\frac{X-\esper[X]}{\sqrt{\var(X)}},\,\,\mathcal{N}(0,1)\right) \leq \frac{18\,KL}{\sigma}.$$
\end{lemma}

\begin{proof}
In the following, we adapt the proof of \cite[Corollary 30]{FMN19}. A classical inequality due to Berry ensures that if $U$ and $V$ are real-valued random variables with Fourier transforms $\Phi_U(\xi)$ and $\Phi_V(\xi)$, and if $m$ is an upper bound on the density of the distribution of $V$ with respect to the Lebesgue measure, then
$$\dkol(U,V) \leq \frac{1}{\pi} \int_{-T}^T \left|\frac{\Phi_U(\xi) - \Phi_V(\xi)}{\xi}\right|\DD{\xi}+ \frac{24m}{\pi T}$$
for any $T>0$; see \cite{Ber41} and \cite[Lemma XVI.3.2]{Fel71}. We set 
$$U = \frac{X-\esper[X]}{\sqrt{\var(X)}} \qquad;\qquad V = \mathcal{N}(0,1)\qquad;\qquad T = \frac{\sigma}{4KL}.$$ 
Note that $\log \Phi_U(\xi)$ is equal to its Taylor series on the interval $(-T,T)$. We have:
$$
\Phi_U(\xi) - \Phi_V(\xi) = \exp\left(\sum_{r\geq 2}\frac{\kappa^{(r)}(X)}{r!} \left(\frac{\I \xi}{\sigma}\right)^{\!r}\right) - \E^{-\frac{\xi^2}{2}} = \E^{-\frac{\xi^2}{2}} \left(\exp\left(\sum_{r\geq 3}\frac{\kappa^{(r)}(X)}{r!}  \left(\frac{\I \xi}{\sigma}\right)^{\!r} \right) - 1\right),
$$
and if $z = \sum_{r\geq 3}\frac{\kappa^{(r)}(X)}{r!}  \left(\frac{\I \xi}{\sigma}\right)^r$, then 
\begin{align*}
|z| \leq K |\xi|^2 \sum_{r \geq 3} \left(\frac{L|\xi|}{\sigma}\right)^{r-2} = |\xi|^2 \frac{KL|\xi|}{\sigma - L|\xi|} \leq \frac{4 K L|\xi|^3}{3\sigma} \leq \frac{|\xi|^2}{3}. 
\end{align*} 
Therefore, $|\exp(z)-1| \leq |z|\,\E^{|z|} \leq \frac{4KL|\xi|^3}{3\sigma}\,\E^{\frac{|\xi|^2}{3}}$, so the integral in Berry's upper bound is smaller than
$$\int_{\R} \frac{4KL|\xi|^2}{3\sigma}\,\E^{-\frac{|\xi|^2}{6}}\DD{\xi} = 4\sqrt{6\pi}\, \frac{KL}{\sigma}.$$
We conclude by using the upper bound $m=\frac{1}{\sqrt{2\pi}}$ on the density of the normal distribution.
\end{proof}     

\begin{proof}[Proof of Theorem \ref{main:berry_esseen}]
The variance of $\majTn$ satisfies 
$$\frac{n^3(1-p_3(\omega\expn))-\frac{3n^2}{2}}{36} \leq \sigma^2 \leq \frac{n^3(1-p_3(\omega\expn))}{36} ;$$
 see our Example \ref{ex:second_cumulant}. In particular, if we suppose that $\max(\frac{\lambda_1\expn}{n},\frac{\lambda_1^{(n),'}}{n}) \leq \frac{1}{2}$ and that $n \geq 4$, then 
$$p_3(\omega\expn) = \int_{x=-1}^1 x^2\,\mu_{\omega\expn}(\!\DD{x}) \leq \frac{1}{4}$$
and $\frac{3}{8}\,\frac{n^3}{36}\leq \sigma^2 \leq \frac{n^3}{36}$ if $n \geq 4$. Now, the cumulants with higher order satisfy 
\begin{align*}
|\kappa^{r}(\majTn)|  &\leq \frac{|B_r|}{r}\, \sum_{i=1}^n i^r \leq \frac{B_r\,(n+\frac{1}{2})^{r+1}}{r(r+1)} \leq r!\, \frac{\pi^2}{100} \, \frac{(n+\frac{1}{2})^{3}}{36} \left(\frac{n+\frac{1}{2}}{2\pi}\right)^{r-2} 
\end{align*}
since 
$$\frac{|B_r|}{r(r+1)} \leq \frac{2\,\zeta(4)\,r!}{20\,(2\pi)^r} = \frac{\pi^2\,r!}{3600}\,\frac{1}{(2\pi)^{r-2}}$$ for $r \geq 4$. The theorem follows from Lemma \ref{lem:berry} with $\sigma^2 = \var(\majTn)$, $L = \frac{n+\frac{1}{2}}{2\pi}$ and
$$ K = \frac{1}{4} \, \frac{(n+\frac{1}{2})^3}{36\,\sigma^2} \geq  \frac{1}{4}.$$
 Note that we could remove the assumption on $\lambda_1\expn$ and $\lambda_1^{(n),'}$ by reworking a bit the argument of Lemma \ref{lem:berry}. By the remark made just after the statement of Theorem \ref{thm:BKS}, the only thing to avoid is that $\limsup_{n \to \infty}(n-\lambda\expn_1)<+\infty$ or $\limsup_{n \to \infty} (n-\lambda_1^{(n),'})<+\infty$, since this prohibits the convergence in distribution to the Gaussian law $\mathcal{N}(0,1)$.
\end{proof}
\medskip

\subsection{Exponential tilting of measures and control of the tilted Fourier transforms}\label{sub:proof_C}
In order to prove Theorem \ref{main:large_deviations}, we shall adapt the proof of \cite[Theorem 4.2.1]{FMN16}, which gives a very similar result in the case where the leading term $\Lambda(z)$ of the scaled log-Laplace transform \emph{is the log-Laplace transform of a non-lattice infinitely divisible distribution}. By our Remark \ref{rem:not_log_laplace}, this is not the case here with $\Lambda_\omega(z)$, so we need to rework some of the arguments.
\medskip

For $x \in (-1,1)$, a simple analysis of the function $h \mapsto \log(\frac{x\sinh \frac{h}{2}}{\sinh \frac{hx}{2}})$ shows that it is defined on the whole real line, even, positive for $h \neq 0$ and strictly convex. Therefore, if $\mu$ is a probability measure on $[-1,1]$ which is not concentrated on $\{-1,1\}$, then the map
$$h \mapsto \int_{t=0}^1 \int_{x=-1}^1 (\phi(th) - \phi(txh))\,\mu(\!\DD{x})\DD{t}$$
is also even, positive for $h \neq 0$ and strictly convex. In particular, this is true when $\mu=\mu_\omega$ and $\omega \in \Omega$ is not one of the two parameters $\omega_1=((1,0,\ldots),(0,\ldots))$ and $\omega_{-1}=((0,\ldots),(1,0,\ldots))$.

\begin{lemma}
For $\omega \in \Omega$, the derivative $\Lambda_\omega'(h)$ goes to $\frac{1}{4}\,\int_{x=-1}^1 (1-x)\,\mu_\omega(\!\DD{x})$ when $h$ goes to infinity.
\end{lemma}
\begin{proof}
We compute 
$$
\Lambda_\omega'(h) = \int_{t=0}^1 \int_{x=-1}^1 \left(\frac{t}{1-\E^{-ht}} - \frac{tx}{1-\E^{-thx}} + \frac{t(x-1)}{2}\right)\mu_\omega(\!\DD{x})\DD{t}.
$$
Hence, the limit when $h$ goes to infinity is $\int_{t=0}^1 \int_{x=-1}^1 \frac{t(1-x)}{2}\,\mu_\omega(\!\DD{x})\DD{t} =\frac{1}{4}\,\int_{x=-1}^1 (1-x)\,\mu_\omega(\!\DD{x}) $.
\end{proof}

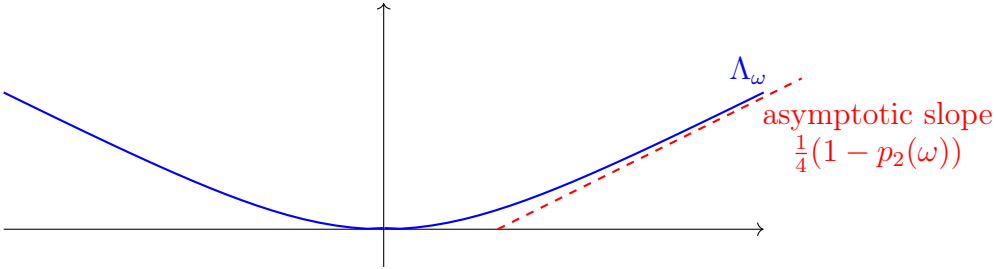
\begin{figure}[ht]
\begin{center}        
\begin{tikzpicture}[scale=1]
\draw [->] (-5,0) -- (5,0);
\draw [->] (0,-0.5) -- (0,3);
\draw [domain=0.01:5,thick,blue] (0,0) plot ({\x},{ln(0.5*(exp(\x)-exp(-\x))/(exp(0.5*\x)-exp(-0.5*\x)))});
\draw [domain=0.01:5,thick,blue] (0,0) plot ({-\x},{ln(0.5*(exp(\x)-exp(-\x))/(exp(0.5*\x)-exp(-0.5*\x)))});
\draw [blue] (4.8,2.1) node {$\Lambda_\omega$};
\draw [dashed,red,thick,shift={(1.5,0)}] (0,0) -- (4,2);
\draw [red] (6.5,1.5) node {asymptotic slope};
\draw [red] (6.5,1) node {$\frac{1}{4}(1-p_2(\omega))$};
\end{tikzpicture}
\caption{Restriction to the real line of the function $\Lambda_\omega$ for $\omega \notin \{\omega_1,\omega_{-1}\}$.}
\end{center}
\end{figure}

Until the end of this section, we fix a convergent sequence $(\lambda\expn)_{n \geq 1}$ with limiting parameter $\omega \notin \{\omega_1,\omega_{-1}\}$, and a real number $y \in (0,\Lambda_\omega'(\infty))$. Since $\Lambda_\omega'$ is strictly increasing, there exists a unique $h \in (0,+\infty)$ such that $\Lambda_\omega'(h)=y$. Moreover, since $\omega \mapsto \Lambda_\omega'$ is continuous with respect to the topology of uniform convergence on compact sets (Lemma \ref{lem:continuity_lambda_psi}), and since $\Lambda_\omega'(\infty) = \frac{1}{4}(1-p_2(\omega))$ is also a continuous function of $\omega$ (Lemma \ref{lem:frobenius_thoma}), for $n$ large enough, there exists a unique parameter $h\expn \in (0,+\infty)$ such that $\Lambda_{\omega\expn}'(h\expn)=y$, and we have $\lim_{n \to \infty} h\expn = h$. \medskip

We define the \emph{tilted} random variable $\tildeX\expn$ by the formula:
$$\proba\!\left[\tildeX\expn \in (x,x+\!\DD{x})\right] = \frac{\E^{h\expn x}}{\esper[\E^{h\expn X\expn}]}\,\proba[X\expn \in (x,x+\!\DD{x})],$$
with $X\expn = \frac{\majTn}{n}$. This exponential tilting of probability measures corresponds to a shift of the log-Laplace transforms: if $\Lambda\expn(z)=\log \esper[\E^{zX\expn}]$ is the log-Laplace transform of $X\expn$, then $\Lambda\expn(z+h\expn)-\Lambda\expn(h\expn)$ is the log-Laplace transform of $\tildeX\expn$. Set 
$$\Psi\expn(z) = \Lambda\expn(z) - z \esper[X\expn] - n\,\Lambda_{\omega\expn}(z);$$ by Theorem \ref{main:log_laplace}, $\Psi\expn(z) = \Psi_{\omega\expn}(z) + o(1)$, and since $(\lambda\expn)_{n \geq 1}$ is a convergent sequence, $\Psi\expn$ converges locally uniformly on $\frac{1}{2}\mathscr{D}_0$ towards the analytic function $\Psi_\omega(z)$. We translate this result for the tilted sequence $(\tildeX\expn)_{n \geq 1}$:
\begin{align*}
\log \esper\!\left[\E^{z\tildeX\expn}\right]
&= n\big(\Lambda_{\omega\expn}(z+h\expn) - \Lambda_{\omega\expn}(h\expn)\big) + \big(\Psi\expn(z+h\expn) - \Psi\expn(h\expn)\big) + z \esper[X\expn] .
\end{align*}
Since
$$\esper\!\left[\tildeX\expn\right] = (\Lambda\expn)'(h\expn) = \esper[X\expn] + n\,\Lambda_{\omega\expn}'(h\expn) + \Psi^{(n),'}(h\expn),$$
we can rewrite the previous estimate as follows:
\begin{align*}
\log \esper\!\left[\E^{z\left(\tildeX\expn- \esper[\tildeX\expn]\right)}\right] 
&=n\big(\Lambda_{\omega\expn}(z+h\expn) - \Lambda_{\omega\expn}(h\expn)-z\Lambda_{\omega\expn}'(h\expn)\big) \\ 
&\quad + \big(\Psi\expn(z+h\expn) - \Psi\expn(h\expn) -z\Psi^{(n),'}(h\expn)\big).
\end{align*}
The two functions $\Lambda_{\omega\expn}(z+h\expn) - \Lambda_{\omega\expn}(h\expn)-z\Lambda_{\omega\expn}'(h\expn)$ and $\Psi\expn(z+h\expn) - \Psi\expn(h\expn) -z\Psi^{(n),'}(h\expn)$ are analytic on the domain $-h\expn + \frac{1}{2}\mathscr{D}_0$, and they converge locally uniformly on the domain  $-h + \frac{1}{2}\mathscr{D}_0$ towards 
$\Lambda_{\omega}(z+h) - \Lambda_{\omega}(h)-z\Lambda_{\omega}'(h)$ and $\Psi_\omega(z+h) - \Psi_\omega(h\expn) -z\Psi_\omega'(h)$. In particular, the convergence holds on the whole imaginary line $\I\R$, because the translation by $-h\expn < 0$ of the domain $\frac{1}{2}\mathscr{D}_0$ moves away from this line the two intervals $\I[2\pi,+\infty)$ and $\I(-\infty,-2\pi]$. 
The estimates above imply then a central limit theorem for the tilted random variable $\tildeX\expn$. The variance of $\tildeX\expn$ is given by the second derivative of $\Lambda\expn$ at $h\expn$, so it is equivalent to $n\,\Lambda_{\omega\expn}''(h\expn)$. Therefore, we can expect that
$$\proba\!\left[\frac{\tildeX\expn-\esper[\tildeX\expn]}{\sqrt{n\,\Lambda_\omega''(h)}}\leq t\right] = \int_{-\infty}^t \E^{-\frac{s^2}{2}}\,\frac{\!\DD{s}}{\sqrt{2\pi}} + o(1).$$
It will be useful to introduce a deformation of the Gaussian distribution, which is a signed measure on $\R$ and turns out to be a better approximation of the law of $\tildeX\expn$. 
\begin{proposition}\label{prop:technical_large_deviations}
Let $F\expn(t)$ be the cumulative distribution function of the variable $\frac{\tildeX\expn-\esper[\tildeX\expn]}{\sqrt{n\,\Lambda_{\omega\expn}''(h\expn)}}$, and 
$$G\expn(t) = \int_{-\infty}^t \left(1 + \frac{\Lambda_\omega'''(h)}{\sqrt{n\,(\Lambda_\omega''(h))^3}}\,\frac{s^3-3s}{6}\right)\E^{-\frac{s^2}{2}}\,\frac{\!\DD{s}}{\sqrt{2\pi}}.$$
The supremum of $|F\expn(t)-G\expn(t)|$ over $\R$ is a $o(n^{-1/2})$.
\end{proposition}

\begin{lemma}\label{lem:difficult}
For any $\delta>0$, $h>0$ and $\omega \in \Omega \setminus \{\omega_1,\omega_{-1}\}$, there exists a constant $A(\delta,h,\omega)>0$ which depends continuously on its three parameters,  such that
$$\forall |\xi|>\delta,\,\,\,\Re ( \Lambda_\omega(h+\I \xi) - \Lambda_\omega(h)) \leq -A(\delta,h,\omega).$$
\end{lemma}

\begin{proof}
Numerical experiments show that given two parameters $h\neq 0$ and $\omega \in \Omega \setminus \{\omega_1,\omega_{-1}\}$, the map $\xi \mapsto  \Re ( \Lambda_\omega(h+\I \xi) - \Lambda_\omega(h))$ stays negative, but is not decreasing on the whole interval $(0,+\infty)$; see Figure \ref{fig:mock_fourier}.
 So, the proof of the lemma will be a bit more subtle than a study of variations. An essential argument which seems clear from Figure \ref{fig:mock_fourier} is that the aforementioned function admits a limit when $\xi$ goes to infinity, and we shall compute this limit in a moment.
\begin{figure}[ht]
  \begin{center}        
  \includegraphics[scale=0.8]{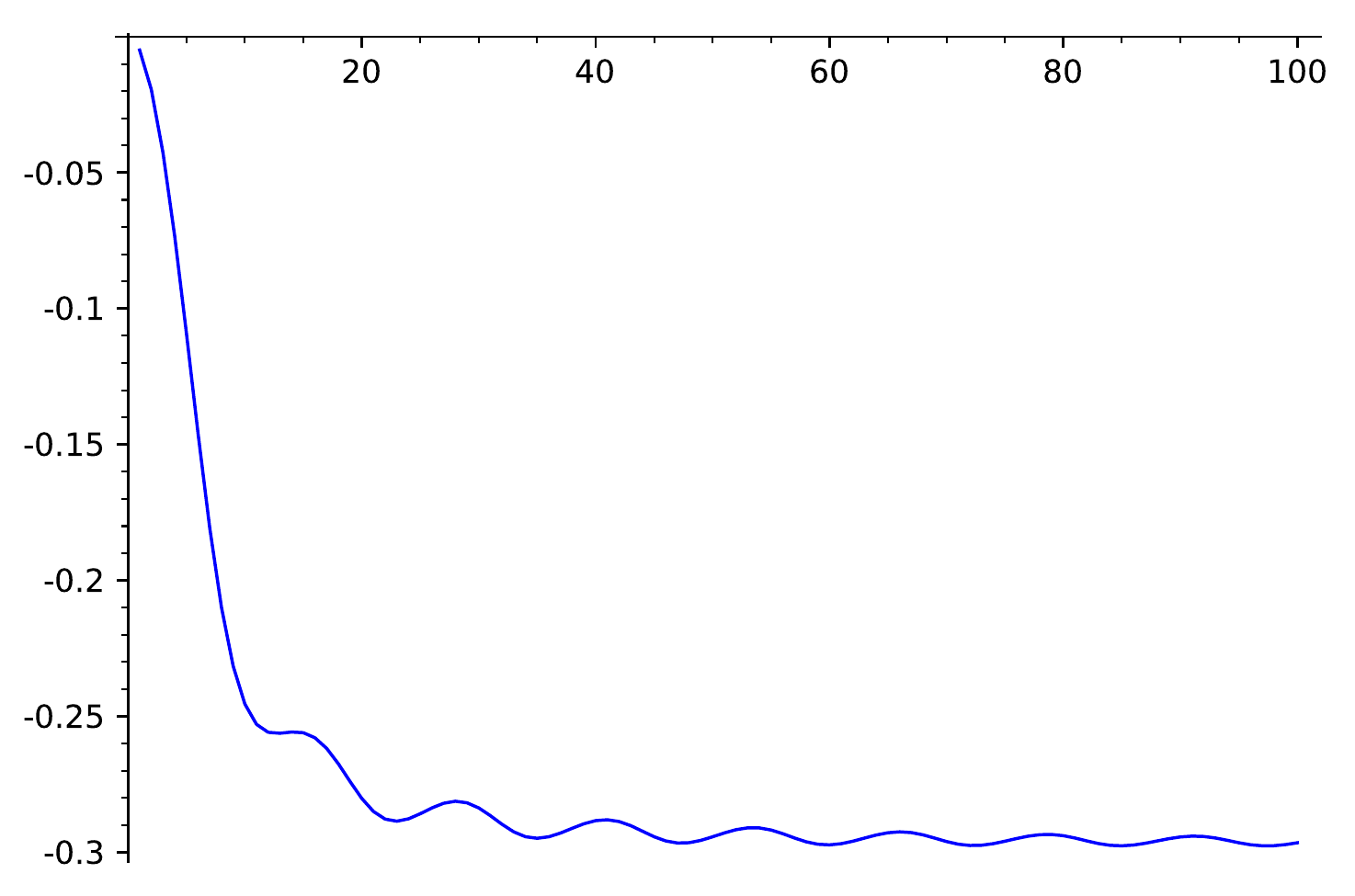}
  \caption{The function $\xi \mapsto \Re ( \Lambda_\omega(h+\I \xi) - \Lambda_\omega(h))$ for $h=5$ and  $\omega=((\frac{1}{2},\frac{1}{2},0,\ldots),(0,\ldots))$.\label{fig:mock_fourier}}
  \end{center}
  \end{figure}
\medskip

If $z = h+\I \xi$, then
\begin{align*}
\Re \log\left( \frac{\sinh \frac{z}{2}}{\frac{z}{2}}\right) &= \frac{1}{2}\log\left|\E^{\frac{h+\I \xi}{2}} - \E^{-\frac{h+\I \xi}{2}}\right|^2 - \frac{1}{2}\log (h^2+\xi^2) \\ 
&= \frac{1}{2}\log(2\cosh h - 2\cos \xi) - \frac{1}{2}\log (h^2+\xi^2) ,
\end{align*}
so
\begin{align*}
\Re \log\left( \frac{\sinh \frac{z}{2}}{\frac{z}{2}}\right) - \Re \log\left( \frac{\sinh \frac{h}{2}}{\frac{h}{2}}\right) &= \frac{1}{2}\log\!\left(1 + \frac{1-\cos \xi}{\cosh h -1}\right) - \frac{1}{2}\log \!\left(1+\frac{\xi^2}{h^2}\right).
\end{align*}
As a consequence, 
\begin{align*}
&\Re ( \Lambda_\omega(h+\I \xi) - \Lambda_\omega(h)) \\ 
& = \frac{1}{2}\int_{t=0}^1\int_{x=-1}^1 \left(\log\!\left(1 + \frac{1-\cos t\xi}{\cosh th -1}\right)  - \log\!\left(1 + \frac{1-\cos tx\xi}{\cosh txh -1}\right) \right) \mu_\omega(\!\DD{x})\DD{t}.
\end{align*}
With $\xi$ and $h$ fixed in $\R \setminus \{0\}$, the function $F( x) = \log(1+\frac{1-\cos x\xi}{\cosh xh - 1})$ is even, and its restriction to $\R_+$ looks as in Figure \ref{fig:decreasing_periodic}.
\begin{figure}[ht]
\begin{center}        
\includegraphics[scale=0.8]{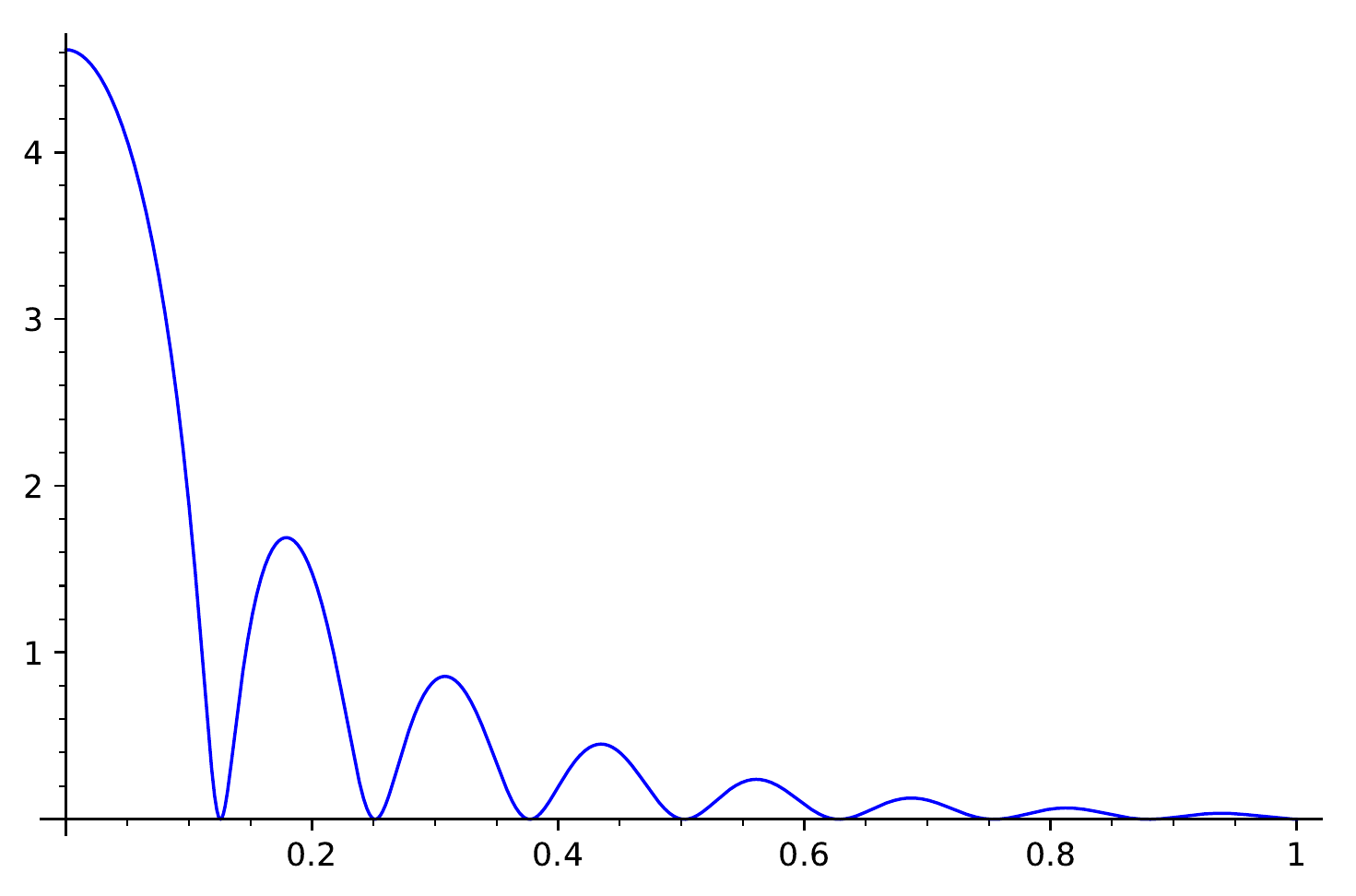}
\caption{Graph of the map $F : x \mapsto \log(1+\frac{1-\cos x\xi}{\cosh xh - 1})$, here with $h=5$ and $\xi=50$.\label{fig:decreasing_periodic}}
\end{center}
\end{figure}
An analysis of functions shows that $F(x)$ is always smaller than its mean $G(x) = \int_{t=0}^1 F(tx)\DD{t}$. Since $G'(x) = \frac{1}{x}(F(x)-G(x))$, this is equivalent to the fact that the map $G$ is decreasing on $\R_+$, see Figure \ref{fig:decreasing_mean}.
\begin{figure}[ht]
\begin{center}        
\includegraphics[scale=0.8]{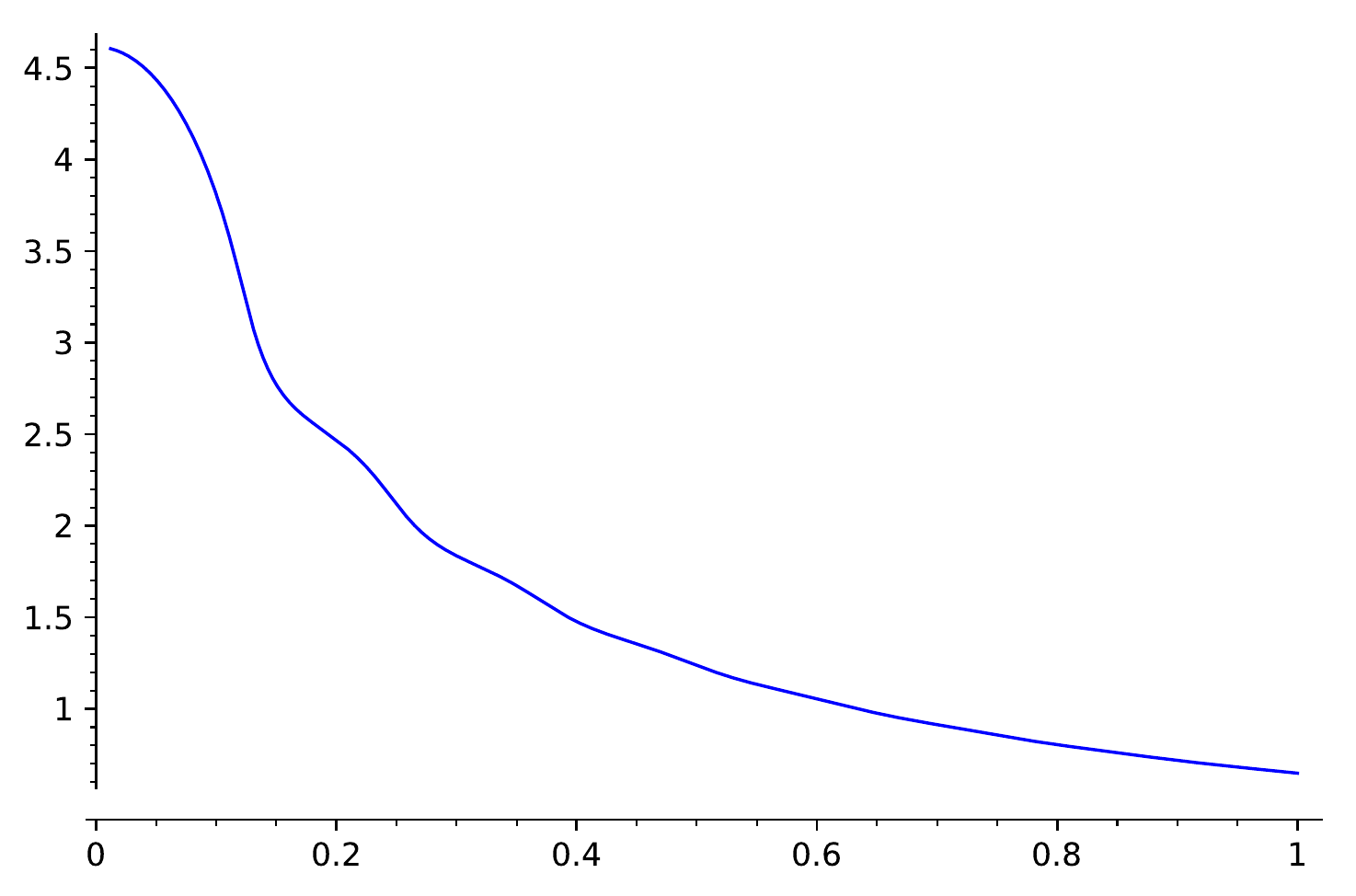}
\caption{Graph of the map $G : x \mapsto \int_{t=0}^1 F(tx)\DD{t}$, again with $h=5$ and $\xi=50$.\label{fig:decreasing_mean}}
\end{center}
\end{figure}
Since $\int_{t=0}^1 (F(t) - F(tx))\DD{t} = G(1)-G(x)$, this implies that if $h \neq 0$, $\xi \neq 0$ and $\mu_\omega$ is not concentrated on $\{-1,1\}$, then $\Re ( \Lambda_\omega(h+\I \xi) - \Lambda_\omega(h))$ is strictly negative.
Then, in order to establish the lemma, it suffices to prove that
$$ \lim_{\xi \to +\infty} \Re ( \Lambda_\omega(h+\I \xi) - \Lambda_\omega(h)) = - A(h,\omega)<0,$$
and that this convergence is locally uniform with respect to the parameters $h$ and $\omega \in \Omega \setminus \{\omega_1,\omega_{-1}\}$. Let us remark that
\begin{align*}
G(x) &= \int_{t=0}^1 \log\!\left(1+\frac{1-\cos tx\xi}{\cosh txh - 1}\right)\DD{t} = \int_{t=0}^1\left( \log\!\left(1 - \frac{\cos tx\xi}{\cosh txh}\right) - \log\!\left(1 - \frac{1}{\cosh txh}\right) \right)\DD{t}  \\ 
&= \sum_{r=1}^\infty \frac{1}{r} \int_{t=0}^1  \frac{1 - (\cos tx\xi)^r}{(\cosh txh)^r}\DD{t}.
\end{align*}
As $\xi$ goes to infinity, the integration of a smooth function against the oscillating weight $(\cos tx\xi)^r$ vanishes if $r$ is odd, and gives $\frac{1}{2^r}\binom{r}{\frac{r}{2}}$ times the integral of the function if $r$ is even. Therefore,
\begin{align*}
\lim_{\xi \to +\infty} G(x) &= \sum_{r=1}^\infty \frac{1}{r} \int_{t=0}^1  \frac{1}{(\cosh txh)^r}\DD{t} - \sum_{s=1}^\infty \frac{1}{2^{2s+1}s} \binom{2s}{s}\int_{t=0}^1  \frac{1}{(\cosh txh)^{2s}}\DD{t} \\ 
&= \int_{t=0}^1 \left(\log\!\left(\frac{1+\sqrt{1-\frac{1}{\cosh^2 txh}}}{2}\right)- \log\!\left(1 - \frac{1}{\cosh txh}\right)\right) \DD{t} \\ 
&= -2\int_{t=0}^1 \log(1-\E^{-t|x|h})\DD{t}. 
\end{align*}
As a consequence,
$$\lim_{\xi \to + \infty} \Re ( \Lambda_\omega(h+\I \xi) - \Lambda_\omega(h)) = \int_{t=0}^1 \int_{x=-1}^1 \log\left(\frac{1-\E^{-t|x|h}}{1-\E^{-th}}\right)\,\mu_\omega(\!\DD{x})\DD{t} = -A(h,\omega)<0,$$
and it is easy to see that this convergence has the right properties of uniformity with respect to the parameters $h \in \R \setminus \{0\}$ and $\omega \in \Omega\setminus\{\omega_1,\omega_{-1}\}$.
\end{proof}

\begin{proof}[Proof of Proposition \ref{prop:technical_large_deviations}]
We shall work with a slightly different distribution on $\R$:
\begin{align*}
\overline{G}\expn(t) &= \int_{-\infty}^t \left(1 + \frac{\Lambda_{\omega\expn}'''(h\expn)}{\sqrt{n\,(\Lambda_{\omega\expn}''(h\expn))^3}}\,\frac{s^3-3s}{6}\right)\E^{-\frac{s^2}{2}}\,\frac{\!\DD{s}}{\sqrt{2\pi}} ;
\end{align*}
this modification is justified by the inequality 
\begin{align*}
|\overline{G}\expn(t)-G\expn(t)| \leq \int_{\R} \left| \frac{\Lambda_{\omega\expn}'''(h\expn)}{\sqrt{(\Lambda_{\omega\expn}''(h\expn))^3}}-\frac{\Lambda_{\omega}'''(h)}{\sqrt{(\Lambda_{\omega}''(h))^3}}\right|\frac{|s|^3+3|s|}{6}\,\E^{-\frac{s^2}{2}}\,\frac{\!\DD{s}}{\sqrt{2\pi n}} = o\!\left(\frac{1}{\sqrt{n}}\right).
\end{align*}
As in the proof of Theorem \ref{main:berry_esseen}, we use the Berry inequality: if 
\begin{align*}
\Phi_1(\xi) &= \esper\!\left[\exp\!\left(\I \xi \,\frac{\tildeX\expn - \esper[\tildeX\expn]}{\sqrt{n\,\Lambda_{\omega\expn}''(h\expn)}}\right)\right]; \\ 
\Phi_2(\xi) &= \int_{-\R} \left(1 + \frac{\Lambda_{\omega\expn}'''(h\expn)}{\sqrt{n\,(\Lambda_{\omega\expn}''(h\expn))^3}}\,\frac{s^3-3s}{6}\right)\E^{-\frac{s^2}{2}+\I s\xi }\,\frac{\!\DD{s}}{\sqrt{2\pi}},
\end{align*}
then for any $T > 0$, 
$$\sup_{t \in \R} \left|F\expn(t) - \overline{G}\expn(t)\right| \leq \frac{1}{\pi}\int_{-T}^T \left|\frac{\Phi_1(\xi)-\Phi_2(\xi)}{\xi}\right|\DD{\xi} + \frac{24m}{\pi T},$$
where $m$ is a uniform upper bound on the (signed) densities $|\overline{G}^{(n),'}(t)|$ with $n \geq 1$ and $t \in \R$ (because of the Gaussian term $\E^{-\frac{s^2}{2}}$ in these densities, there exists indeed such a constant $m>0$). Let us analyse the two Fourier transforms $\Phi_1$ and $\Phi_2$ on an interval $I = [-T,T] = [-\Delta \sqrt{n},\Delta \sqrt{n}]$. First, $\Phi_2$ is the following explicit function:
$$\Phi_2(\xi) = \left(1 + \frac{\Lambda_{\omega\expn}'''(h\expn)}{\sqrt{n\,(\Lambda_{\omega\expn}''(h\expn))^3}}\,\frac{(\I \xi)^3}{6}\right)\E^{-\frac{\xi^2}{2}}.$$
Indeed, $s^3-3s$ is the Hermite polynomial $H_3(s)$, and one has the general formula
$$\int_{\R} H_k(s)\,\E^{-\frac{s^2}{2}+\I s\xi}\,\frac{\!\DD{s}}{\sqrt{2\pi}} = (\I \xi)^k\,\E^{-\frac{\xi^2}{2}};$$
see \cite[Chapter 5]{Sze39}. On the other hand, on a small interval $[-\delta \sqrt{n},\delta \sqrt{n}] = I' \subset I$, with $z =\frac{\I \xi}{\sqrt{n\,\Lambda_{\omega\expn}''(h\expn)}}$, we have by Taylor approximation:
\begin{align*}
n \big(\Lambda_{\omega\expn}(z+h\expn) - \Lambda_{\omega\expn}(h\expn)-z\Lambda_{\omega\expn}'(h\expn) \big) &= -\frac{\xi^2}{2} + \frac{\Lambda_{\omega\expn}'''(h\expn) \,(\I \xi)^3}{6\,\sqrt{n\,(\Lambda_{\omega\expn}''(h\expn))^3}} (1+o_\delta(1)); \\ 
\big(\Psi\expn(z+h\expn) - \Psi\expn(h\expn) -z\Psi^{(n),'}(h\expn)\big) &= -\frac{\Psi^{(n),''}(h\expn)\,\xi^2}{2n\,\Lambda_{\omega\expn}''(h\expn)}\,(1+o_\delta(1)).
\end{align*}
Therefore,
$$\Phi_1(\xi) = \left(1 + \frac{\Lambda_{\omega\expn}'''(h\expn)}{\sqrt{n\,(\Lambda_{\omega\expn}''(h\expn))^3}}\,\frac{(\I \xi)^3}{6}\,(1+o_\delta(1))-\frac{\Psi^{(n),''}(h\expn)\,\xi^2}{2n\,\Lambda_{\omega\expn}''(h\expn)}\,(1+o_\delta(1))\right)\E^{-\frac{\xi^2}{2}},$$
and we have
$$ \int_{I'} \left|\frac{\Phi_1(\xi)-\Phi_2(\xi)}{\xi}\right| \DD{\xi} \leq \int_{I'} \left(O_\delta\!\left(\frac{|\xi|}{n}\right) + o_\delta\!\left(\frac{|\xi|^2}{\sqrt{n}}\right)\right)\E^{-\frac{\xi^2}{2}}\DD{\xi} = o_\delta\!\left(\frac{1}{\sqrt{n}}\right),$$
by using the uniform convergence around $h$ of the functions $\Lambda_{\omega\expn}$ and $\Psi\expn$ and their derivatives. Consequently, in order to end the proof of the lemma, it suffices now to control $\Phi_1(\xi)$ and $\Phi_2(\xi)$ over $I \setminus I'$. If we write
$$\Phi_1(\xi) = \underbrace{\E^{n\big(\Lambda_{\omega\expn}(z+h\expn) - \Lambda_{\omega\expn}(h\expn)-z\Lambda_{\omega\expn}'(h\expn) \big)}}_{a(\xi)}\,\,\underbrace{\E^{\big(\Psi\expn(z+h\expn) - \Psi\expn(h\expn) -z\Psi^{(n),'}(h\expn)\big)}}_{b(\xi)},$$
then Lemma \ref{lem:difficult} shows that $a(\xi)$ is bounded over $I \setminus I'$ by $\E^{-n A(\delta)}$ for some constant $A(\delta)>0$. On the other hand, the quantity $b(\xi)$ is bounded over $I$ by some constant $B(\Delta)$. So,
$$\int_{I \setminus I'} \left|\frac{\Phi_1(\xi)}{\xi}\right| \DD{\xi} \leq 2B(\Delta)\E^{-n A(\delta)} \int_{\delta \sqrt{n}}^{\Delta\sqrt{n}} \frac{\DD{\xi}}{\xi} = 2B(\Delta)\,\E^{-n A(\delta)}\,\log\!\left(\frac{\Delta}{\delta}\right).$$
A similar bound exists for the integral of $\left|\frac{\Phi_2(\xi)}{\xi}\right|$, since the Gaussian term $\E^{-n \frac{\xi^2}{2}}$ is bounded from above by $\E^{-n\frac{\delta^2}{2}}$ on $I \setminus I'$. We conclude that for any pair $\delta<\Delta$,
$$\sup_{t \in \R} \left|F\expn(t) - \overline{G}\expn(t)\right| \leq C(\Delta)\,\log\!\left(\frac{\Delta}{\delta}\right)\,\E^{-n A(\delta)} + \frac{24m}{\pi \Delta \sqrt{n}}+ o_\delta\!\left(\frac{1}{\sqrt{n}}\right) $$
for some positive constants $A(\delta)$ and $C(\Delta)$.
Given $\eps>0$, we choose $\delta$ small enough and $\Delta$ large enough so that the two last terms of the right-hand side are smaller than $\frac{\eps}{\sqrt{n}}$ for $n$ large enough. Then, the exponential term goes to $0$ faster than $\frac{1}{\sqrt{n}}$, so it is also smaller than $\frac{\eps}{\sqrt{n}}$ for $n $ large enough. Thus, for any $\eps>0$,
$$\sup_{t \in \R} \left|F\expn(t) - \overline{G}\expn(t)\right| \leq \frac{3\eps}{\sqrt{n}}$$
for $n$ large enough, which means that the distance between the two distributions $F\expn$ and $\overline{G}\expn$ is a $o(n^{-1/2})$. 
\end{proof}

\begin{proof}[Proof of Theorem \ref{main:large_deviations}]
Our proof follows now the same argument as in \cite[Theorem 4.2.1]{FMN16} and \cite[Theorem A]{MN22}, and is inspired by standard techniques from the theory of large deviations, in particular the Bahadur--Rao estimates for sums of non-lattice distributed i.i.d.~random variables \cite{BR60}. We have:
\begin{align*}
&\proba[\majTn - \esper[\majTn] \geq yn^2] \\ 
&= \proba\!\left[X\expn \geq \esper[X\expn] + yn\right] \\ 
&= \esper\left[\E^{h\expn X\expn}\right] \int_{\esper[X\expn] + yn}^\infty \E^{-h\expn x}\,\,\proba\!\left[\tildeX\expn \in (x,x+dx)\right] \\ 
&= \E^{-n (\Lambda_{\omega\expn})^*(y)  +\Psi\expn(h\expn) - h\expn \Psi^{(n),'}(h\expn)} \int_{- \Psi^{(n),'}(h\expn)}^\infty \E^{-h\expn x}\,\,\proba\!\left[\tildeX\expn - \esper[\tildeX\expn] \in (x,x+dx)\right] \\ 
&= \E^{-n (\Lambda_{\omega\expn})^*(y)  +\Psi\expn(h\expn) - h\expn \Psi^{(n),'}(h\expn)} \int_{- \frac{\Psi^{(n),'}(h\expn)}{\sqrt{n\,\Lambda_{\omega\expn}''(h\expn)}}}^\infty \E^{-h\expn \sqrt{n\,\Lambda_{\omega\expn}''(h\expn)}\,x}\,dF\expn(x).
\end{align*}
Let $I\expn$ be the integral on the last line, and $$b\expn=-\frac{\Psi^{(n),'}(h\expn)}{\sqrt{n\,\Lambda_{\omega\expn}''(h\expn)}}$$ its lower bound of integration, which goes to $0$ as $n$ goes to infinity. By using an integration by parts and the estimate from Proposition \ref{prop:technical_large_deviations}, we get
\begin{align*}
 I\expn &= h\expn \sqrt{n\,\Lambda_{\omega\expn}''(h\expn)} \int_{b\expn}^\infty \E^{-h\expn \sqrt{n\,\Lambda_{\omega\expn}''(h\expn)}\,x}\,(F\expn(x)-F\expn(b\expn))\,\DD{x} \\ 
 &=  h\expn \sqrt{n\,\Lambda_{\omega\expn}''(h\expn)} \int_{b\expn}^\infty \E^{-h\expn \sqrt{n\,\Lambda_{\omega\expn}''(h\expn)}\,x}\,(G\expn(x)-G\expn(b\expn))\,\DD{x} \\ &\quad 
 + o\!\left(h\expn \sqrt{\Lambda_{\omega\expn}''(h\expn)} \int_{b\expn}^\infty \E^{-h\expn \sqrt{n\,\Lambda_{\omega\expn}''(h\expn)}\,x}\,\DD{x}\right)\\ 
 &= \int_{b\expn}^\infty \E^{-h\expn \sqrt{n\,\Lambda_{\omega\expn}''(h\expn)}\,x}\,\,dG\expn(x) + o\!\left(\frac{1}{\sqrt{n}}\right).
\end{align*}
In the integral against the (signed) distribution $dG\expn(x)$, the main contribution comes from the Gaussian term and is equal to
\begin{align*}
\int_{b\expn}^\infty \E^{-h\expn \sqrt{n\,\Lambda_{\omega\expn}''(h\expn)}\,x-\frac{x^2}{2}} \,\frac{\!\DD{x}}{\sqrt{2\pi}} &= \E^{\frac{(h\expn)^2 \,n\,\Lambda_{\omega\expn}''(h\expn)}{2}}\int_{b\expn + h\expn \sqrt{n\,\Lambda_{\omega\expn}''(h\expn)}} \E^{-\frac{y^2}{2}}\,\frac{\!\DD{y}}{\sqrt{2\pi}} \\ 
&= \frac{\E^{h\expn\,\Psi^{(n),'}(h\expn) }}{h\expn\,\sqrt{2\pi n\,\Lambda_{\omega\expn}''(h\expn) }}\,(1+o(1));
\end{align*}
the other terms give a contribution of order $O(n^{-1})=o(n^{-1/2})$. Therefore,
$$\proba[\majTn - \esper[\majTn] \geq yn^2]  = \frac{\E^{-n (\Lambda_{\omega\expn})^*(y)  +\Psi\expn(h\expn)} }{h\expn\,\sqrt{2\pi n\,\Lambda_{\omega\expn}''(h\expn)}}(1+o(1)),$$
and we conclude by using the convergence of $h\expn$ towards $h$, of $\Psi\expn$ towards $\Psi_\omega$ and of $\Lambda_{\omega\expn}''$ towards $\Lambda_\omega''$.
\end{proof}

\bigskip
\bigskip

\printbibliography

\end{document}